\documentclass[12pt]{amsart}
\usepackage[utf8]{inputenc} 
\usepackage[active]{srcltx}
\usepackage{a4wide}
\usepackage{amsthm,amsfonts,amsmath,mathrsfs,amssymb}
\usepackage{dsfont}
\usepackage{mathtools}
\usepackage[T1]{fontenc}
\usepackage[utf8]{inputenc}
\usepackage{enumerate}
\usepackage{comment}
\usepackage[left=4cm,top=4cm,right=4cm,bottom=4cm]{geometry}
\numberwithin{equation}{section}

\usepackage{geometry}
\usepackage{graphicx}
\usepackage{amssymb}
\usepackage{amsmath}
\usepackage{amsthm}
\usepackage{listings}
\usepackage[colorlinks=true,urlcolor=blue,citecolor=blue,linkcolor=blue]{hyperref}
\usepackage{esint}
\usepackage{xcolor}

\usepackage{etex}
\usepackage[all]{xy}
\usepackage{pgf,tikz}
\usetikzlibrary{arrows}
\usetikzlibrary{patterns}
\usepackage{hyperref}

\newcommand{\R}{{\mathbb R}} 
\newcommand{\C}{{\mathbb C}} 
\newcommand{\N}{{\mathbb N}} 
\newcommand{\Z}{{\mathbb Z}}

\newcommand{\D}{{\mathbb D}}

\renewcommand{\Re}{\mathrm{Re}}

\RequirePackage{etex}

\newtheorem{theorem}{Theorem}[section]
\newtheorem{lemma}[theorem]{Lemma}
\newtheorem{proposition}[theorem]{Proposition}
\newtheorem{corollary}[theorem]{Corollary}

\theoremstyle{definition}
\newtheorem{definition}[theorem]{Definition}
\newtheorem{example}[theorem]{Example}

\theoremstyle{remark}
\newtheorem{remark}[theorem]{Remark}

\numberwithin{equation}{section}

\makeatletter
\@namedef{subjclassname@2020}{%
  \textup{2020} Mathematics Subject Classification}
\makeatother

\begin{document}
\title[Semigroups of composition operators]{Composition operators on the algebra of Dirichlet series}

\author[M.D. Contreras]{Manuel D. Contreras}
\address{Departamento de Matem\'atica Aplicada II and IMUS, Escuela T\'ecnica Superior de Ingenier\'ia, Universidad de Sevilla,
	Camino de los Descubrimientos, s/n 41092, Sevilla, Spain}
\email{contreras@us.es}

\author[C. G\'omez-Cabello ]{Carlos G\'omez-Cabello}
\address{Departamento de Matem\'atica Aplicada II and IMUS, Edificio Celestino Mutis, Avda. Reina Mercedes, s/n. 41012 - Sevilla, Universidad de Sevilla,
Sevilla, Spain}
\email{cgcabello@us.es}

\author[L. Rodr\'iguez-Piazza]{Luis Rodr\'iguez-Piazza}
\address{Departmento de An\'alisis Matem\'atico and IMUS, Facultad de Matem\'aticas, Universidad
	de Sevilla, Calle Tarfia, s/n 41012 Sevilla, Spain}
\email{piazza@us.es}


\date{\today}


\thanks{This research was supported in part by Ministerio de Ciencia e Innovación, Spain,  and the European Union (FEDER), project PID2022-136320NB-I00, and Junta de Andaluc{\'i}a, FQM133 and FQM104.}

\maketitle 
\begin{abstract}
    The algebra of Dirichlet series $\mathcal{A}(\C_+)$ consists on those Dirichlet series convergent in the right half-plane $\C_+$ and which are also uniformly continuous there. This algebra was recently introduced by Aron, Bayart, Gauthier, Maestre, and Nestoridis. We describe the symbols $\Phi:\C_+\to\C_+$ giving rise to bounded composition operators $C_{\Phi}$ in $\mathcal{A}(\C_+)$ and denote this class by $\mathcal{G}_{\mathcal{A}}$. We also characterise when the operator $C_{\Phi}$ is compact in $\mathcal{A}(\C_+)$. As a byproduct, we show that the weak compactness is equivalent to the compactness for $C_{\Phi}$. Next, the closure under the local uniform convergence of several classes of symbols of composition operators in Banach spaces of Dirichlet series is discussed. We also establish a one-to-one correspondence between continuous semigroups of analytic functions $\{\Phi_t\}$ in the class $\mathcal{G}_{\mathcal{A}}$ and strongly continuous semigroups of composition operators $\{T_t\}$, $T_tf=f\circ\Phi_t$, $f\in\mathcal{A}(\C_+)$. We conclude providing examples showing the differences between the symbols of bounded composition operators in $\mathcal{A}(\C_+)$ and the Hardy spaces of Dirichlet series $\mathcal{H}^p$ and $\mathcal{H}^{\infty}$.
\end{abstract}
\tableofcontents

\section{Introduction}
Since the introduction of the first Hardy spaces of Dirichlet series by Hendemalm, Linqvist and Seip in \cite{HedLinSeip}, these spaces have attired an increasing attention. Nonetheless, when it comes to the algebra of Dirichlet series $\mathcal{A}(\C_+)$, this is, the closure of Dirichlet polynomials in the $\mathcal{H}^{\infty}$-norm, the literature is significantly lesser. To the best of the authors' knowledge, this new space of Dirichlet series was first presented in \cite{Aron-Bayart-cia}, where some of its most essential properties were proven as well as some non-trivial related results. To mention some others works concerning $\mathcal{A}(\C_+)$, see  \cite{sargenta}, where the validity of Carlson's identity in $\mathcal{A}(\C_+)$ is studied, and \cite{carandoetal}, where the authors consider the maximal Bohr's strip.

\ 
One of the most recurrent questions that arises when studying Banach spaces of analytic functions is that of the boundedness of composition operators. Namely, given a Banach space $X$ of analytic functions in a domain $\Omega$ and a function $\Phi:\Omega\to\Omega$, we consider the operator $C_{\Phi}$ given by $C_{\Phi}f=f\circ\Phi$, for $f\in X$. One of the goals is to give necessary and sufficient conditions on the \emph{symbol} $\Phi$ so that $f\circ\Phi$ belongs to $X$, for all $f\in X$, and that $C_{\Phi}$ defines a bounded operator from $X$ into itself.

\ 
In the context of Banach spaces of Dirichlet series, the first characterisation of bounded composition operators was given for the Hilbert space $\mathcal{H}^2$ by
Gordon and Hedenmalm in \cite{gorheda}. To do so, they introduced the now known-as \emph{Gordon-Hedenmalm class} $\mathcal{G}$. Little after, Bayart  obtained in \cite{bayarto} results regarding the boundedness of composition operators on the spaces $\mathcal H^{p}$ for $1\leq p\leq +\infty$. The complete characterisation of the boundedness for $p\notin 2\N \cup\{ +\infty\}$ remains a challenging open problem.  In the same work, Bayart also characterised the boundedness of composition operators in $\mathcal{H}^{\infty}$ giving rise to a class of symbols that we denote by $\mathcal{G}_{\infty}$. We recall both the definitions of $\mathcal{G}$ and $\mathcal{G}_{\infty}$ right straightaway. As customary, for $\theta\in\R$, we denote
$
\C_{\theta}:=\{s\in\C:\text{Re}(s)>\theta\}
$
and  $\C_{+}=\C_{0}$. 
\begin{definition}\label{def: gorhed} Let $\Phi:\C_{+}\to\C_{+}$ be an analytic function.
\begin{enumerate}
\item We say that $\Phi$ belongs to the class $\mathcal{G}_{\infty}$ if there exist $c_\Phi\in\N\cup\{0\}$ and a convergent Dirichlet series $\varphi$ such that
    \begin{equation*}
        \Phi(s)=c_{\Phi}s+\varphi(s).
    \end{equation*}
The value $c_{\Phi}$ is known as the {\sl characteristic} of the function $\Phi$. 
\item We say that $\Phi$ belongs to the {\sl Gordon-Hedenmalm class} $\mathcal{G}$ if $\Phi\in  \mathcal{G}_{\infty}$ and  $\Phi(\C_+)\subset \C_{1/2}$ in case $c_{\Phi}=0$.
\end{enumerate}
\end{definition}
In Section \ref{sec:compopalgb}, after recalling some basic notions related to Dirichlet series, we describe the symbols giving rise to bounded composition operators in the context of the algebra of Dirichlet series $\mathcal{A}(\C_+)$. This leads us to the introduction of the family of symbols $\mathcal{G}_{\mathcal{A}}$:
\begin{definition}\label{def: ga}
Let $\Phi:\C_+\to\C_+$. We say that $\Phi$ belongs to the  class $\mathcal{G}_{\mathcal{A}}$ if $\Phi\in\mathcal{G}_{\infty}$ and, for every $M>0$, $\Phi$ is uniformly continuous in 
\[
A_M:=\{s\in\C_+:0<\text{Re}(\Phi(s))<M\}.
\]
\end{definition}
Then, the following result is established:
\begin{theorem}[See Theorem \ref{thsimal}]
Let $\Phi:\C_+\to\C_+$ be analytic. Then, $C_{\Phi}$ defines a bounded composition operator in $\mathcal{A}(\C_+)$ if and only if $\Phi\in\mathcal{G}_{\mathcal{A}}$.
\end{theorem}
Furthermore, we show that there are non-uniformly continuous symbols in $\C_+$ belonging to the class $\mathcal{G}_{\mathcal{A}}$. We also prove that $\mathcal{G}_{\mathcal{A}}\subsetneq\mathcal{G}_{\infty}$.

\ 
Section \ref{sec:compactness} is devoted to the description of compact composition operators in the algebra $\mathcal{A}(\C_+)$, as well as the weakly-compact ones (see Theorem \ref{thcomp}). In fact, we prove that these classes of operators coincide. Moreover, we show that the compactness in $\mathcal{A}(\C_+)$ is equivalent to not fixing a copy of $c_0$. The compact composition operators in $\mathcal{H}^{\infty}$ were described by Bayart in \cite{bayarto}. It is also worth mentioning that Lefèvre proved in \cite{pascal} that they coincide with the weakly compact composition operators on $\mathcal{H}^{\infty}$. We go beyond this equivalence proving that every non-compact composition operators on $\mathcal{H}^{\infty}$ fixes a copy of $\ell^{\infty}$ (see Theorem \ref{thcompa}).

\ 
Throughout this work and in connection with the boundedness of composition operators on Banach spaces of Dirichlet series, we will be treating with several classes of Dirichlet series symbols. In Section \ref{sec:closure-classes}, the closure of the three classes $\mathcal{G}$, $\mathcal{G}_{\infty}$ and $\mathcal{G}_{\mathcal{A}}$ under the local uniform convergence in $\C_+$ is studied (see Theorem \ref{closeness} and Theorem \ref{th: ga no cerrada}). 
 
\
Once the description of bounded composition operators is settled, in section \ref{sec:semigroupsinalg} we focus our attention on the study of semigroups of composition operators on the algebra $\mathcal{A}(\C_+)$. Berkson and Porta's work \cite{porta} initiated the study of such semigrous in the context of the classical Hardy space of the unit disc $H^2(\D)$. There, they proved that a semigroup of composition operators $\{C_{\Phi_{t}}\}$ is strongly continuous in $H^2(\D)$ if and only if the family of functions  $\{\Phi_{t}\}$ is a continuous semigroup of holomorphic functions in the unit disc. Thereafter, several works studying these semigroups in other contexts such has the Bergman spaces, the disc algebra, the Bloch space or BMOA arose. See \cite{Anderson}, \cite{Arevalo-Contreras-Piazza}, \cite{Avicou}, \cite{Betsakos}, \cite{BCDMPS}, \cite{Gallardo} and references therein. An analogue version of Berkson and Porta's theorem in the Hardy spaces of Dirichlet series $\mathcal{H}^p$ for $1\leq p<\infty$ was established in \cite{noi}. More specifically, the semigroups $\{C_{\Phi_{t}}\}$ are strongly continuous if and only if the semigroup $\{\Phi_t\}$ of symbols in $\mathcal{G}$ is continuous. Surprisingly, for these continuous semigroups, the convergence on compact subsets of $\C_+$ happens to be equivalent to the convergence in half-planes $\C_{\varepsilon}$ (\cite[Theorem 4.6]{noi}). Afterwards, an even more striking result was found in \cite[Theorem 1.2]{noi1}: every continuous semigroup $\{\Phi_{t}\}$ in the Gordon-Hedenmalm class $\mathcal{G}$ converges to the identity map uniformly in the right half-plane as $t$ tends to $0$. A similar result was proven in the context of the unit disc algebra $A(\D)$.

\ 
In section \ref{sec:semigroupsinalg}, using that a continuous semigroup $\{\Phi_t\}$ in the class $\mathcal{G}_{\infty}$ is of the form $\Phi_t(s)=s+\varphi_t(s)$, $\varphi_t\in\mathcal{D}$, (see \cite[Proposition 3.2]{noi}) the correspondence between strongly continuous semigroups of composition operators in $\mathcal{A}(\C_+)$ and continuous semigroups of symbols in the class $\mathcal{G}_{\mathcal{A}}$ is settled in Theorem \ref{teor:fcont-uniform}. 

\ 
In Section \ref{sec:semigroupsinalg}, some properties of the Koenigs functions of the continuous semigroups in the class $\mathcal{G}_{\mathcal{A}}$ are studied. These are used to provide some interesting examples of continuous semigroups in $\mathcal{G}_{\infty}$. For instance, semigroups in the class $\mathcal{G}_{\infty}$ which are not in the class $\mathcal{G}_{\mathcal{A}}$ (see Theorem \ref{th: sem ginf no ga}). 

\ 
Section \ref{sec:semanfunc} is devoted to recall some notions related to continuous semigroups of analytic functions and their infinitesimal generators. Moreover, we also state those theorems from \cite{noi} and \cite{noi1} that will be needed to prove the results of the subsequent sections.

\section{The algebra of Dirichlet series $\mathcal{A}(\C_+)$ and its composition operators}\label{sec:compopalgb}
\subsection{Dirichlet series}
We denote by $\mathcal{D}$ the space of convergent Dirichlet series. That is, the family of series
$$
\varphi(s)=\sum_{n=1}^{\infty}a_nn^{-s}, \quad \{a_n\}\in\C
$$
which converge in a certain half-plane $\C_{\theta}$, $\theta\in\R$. If there are only finitely many non-zero coefficients $a_n$, we say that $\varphi$ is a \emph{Dirichlet polynomial}. If there is only one non-zero coefficient, we then have a \emph{Dirichlet monomial}. In the context of Dirichlet series, the analogue to the radii of convergence of the analytic functions are the abscissae of convergence. The most relevant ones are the following:
$$
\sigma_{c}(\varphi)=\inf\{ \text{Re}(s):  \sum_{n=1}^{\infty}a_nn^{-s} \textrm{ is convergent}\};
$$
$$
\sigma_{b}(\varphi)=\inf\{ \sigma:  \sum_{n=1}^{\infty}a_nn^{-s} \textrm{ is bounded on } \C_{\sigma}\};
$$
$$
\sigma_{u}(\varphi)=\inf\{ \sigma:  \sum_{n=1}^{\infty}a_nn^{-s} \textrm{ is uniformly convergent on } \C_{\sigma}\};
$$
$$
\sigma_{a}(\varphi)=\inf\{ \text{Re}(s):  \sum_{n=1}^{\infty}a_nn^{-s} \textrm{ is absolutely convergent}\}.
$$
In \cite{bohr}, H. Bohr proved that for every Dirichlet series $\varphi$, it holds that $\sigma_u(\varphi)=\sigma_b(\varphi)$. Moreover, there exists the following relation between the abscissae defined above
$$
\sigma _{c}(\varphi)\leq \sigma _{b}(\varphi)=\sigma _{u}(\varphi)\leq \sigma _{a}(\varphi)\leq \sigma _{c}(\varphi)+1.
$$
A proof of these inequalities can be found in either \cite[Section 4.2]{queffelecs} or \cite[Chapter 1]{peris}.

There are some basic results concerning Dirichlet series that will arise throughout the exposition. We state them straightaway. The reader may find a proof of the first one in \cite[Theorem 8.4.1]{queffelecs}. The second one, which is a Montel's-type principle for Dirichlet series, was first established in \cite[Lemma 18]{bayarto} (see also \cite[Theorem 6.3.1]{queffelecs}).
\begin{theorem}\label{queffelecs}
Let $\sigma,\nu\in\R$. Consider $\varphi:\C_{\sigma}\to\C_{\nu}$ analytic such that it can be written as a  Dirichlet series in a certain half-plane. Then, $\sigma_u(\varphi)\leq\sigma $. In particular, $\varphi$ is bounded in $\C_{\sigma+\varepsilon}$ for all $\varepsilon>0$.
\end{theorem}
An improved version of this theorem can be found in \cite[Theorem 1]{brekou}, where the authors show that $\C_{\nu}$ can be replaced by  $\C\setminus\{\alpha,\beta\}$, $\alpha\not=\beta$.
\begin{corollary}\label{coro:cont unif}
Let $\varphi$ be a Dirichlet series in the conditions of the prior theorem. Then, $\varphi$ is uniformly continuous in every half-plane $\C_{\sigma+\varepsilon}$ for all $\varepsilon>0$.
\end{corollary}
 \begin{theorem}\label{th: bayart}
Let $\{f_n\}$ be a bounded sequence in $\mathcal{H}^{\infty}$. Then, there exist both a subsequence $\{f_{n_k}\}$  and a function $f\in\mathcal{H}^{\infty}$, such that $f_{n_k}$ converges uniformly to $f$ on each half-plane $\C_{\varepsilon}$, for all $\varepsilon>0$.
\end{theorem}
The following theorem will also be useful in the proof of several results in forthcoming sections.

\begin{theorem}\label{lemma: bay-cas}
 Suppose that $\Phi:\C_+\to\C$ is a function such that $k^{-\Phi}\in\mathcal{H}^{\infty}$ for
every $k\in\N$. Then $\Phi$ is analytic in $\C_+$ and there exist $c_0\in\N\cup\{0\}$, $\varphi\in\mathcal{D}$ such that $\varphi$
extends to $\C_+$ and $\Phi(s) = c_0s + \varphi(s)$ for all $s\in\C_+$.
\end{theorem}
The statement was implicit in Gordon and Hedenmalm's seminal work \cite{gorheda}. Years later, it was explicitly stated in \cite[Lemma 1]{bayart-castillo} (see also \cite[Theorem 8.3.1]{queffelecs}). 

\subsection{The algebra of Dirichlet series} 
In the present work we shall focus our attention in the algebra of Dirichlet series $\mathcal{A}(\C_+)$. We recall that the space $\mathcal{H}^{\infty}$ consists on those Dirichlet series $\varphi$ converging in $\C_+$ such that
\[
\|\varphi\|_{\infty}=\sup_{s\in\C_+}|\varphi(s)|<\infty.
\]
In fact, thanks to Bohr's theorem, a Dirichlet series convergent in $\C_{\theta}$ for some $\theta>0$ and having a bounded analytic extension to $\C_+$, belongs to the space $\mathcal{H}^{\infty}$. It is also useful to introduce the following family of Banach spaces of Dirichlet, namely, $\mathcal{H}^{\infty}(\C_{\varepsilon})$, $\varepsilon>0$. These spaces consist on those Dirichlet series $\varphi$ converging in $\C_{\varepsilon}$ such that
\[
\|\varphi\|_{\mathcal{H}^{\infty}(\C_{\varepsilon})}:=\sup_{s\in\C_{\varepsilon}}|\varphi(s)|<\infty.
\]
Observe that if we take $\varepsilon=0$ we recover the space $\mathcal{H}^{\infty}$ from above. 
Now, let us give the definition the algebra of Dirichlet series $\mathcal{A}(\C_+)$ and present some of its most basic properties.
\begin{definition}
We define the algebra of Dirichlet series, denoted by $\mathcal{A}(\C_+)$, as the collection of all Dirichlet series
\[
f(s)=\sum_{n=1}^{\infty}a_nn^{-s}
\]
convergent in $\C_+$ and such that they define a uniformly continuous function $f$ on $\C_+$.
\end{definition}
The following theorem contains some basic properties of the algebra $\mathcal{A}(\C_+)$, see \cite[Proposition 2.2, Theorem 2.3]{Aron-Bayart-cia} for a proof.
\begin{theorem}\label{thpropbasa}
The algebra of Dirichlet series $\mathcal{A}(\C_+)$ satisfies the following properties:
\begin{itemize}
    \item[(1)] $\mathcal{A}(\C_+)$ is a closed subspace of $\mathcal{H}^{\infty}$ .
    \item[(2)]$
    \mathcal{A}(\C_+)=\overline{\emph{span}\{n^{-s}:n\in\N\}}^{\|\cdot\|_{\infty}}
   $.
\end{itemize}
\end{theorem}
Consequently, the algebra $\mathcal{A}(\C_+)$ is a Banach algebra.

\ 

The proof of Proposition 2.2 in \cite{Aron-Bayart-cia} yields the following more general result on the characterisation of the membership to $\mathcal{A}(\C_+)$. We include the proof to recall the ideas  used there.
\begin{proposition}\label{thprop}
    Let $\varphi$ be a uniformly continuous function in $\C_+$, analytic there and such that $\varphi\in\mathcal{D}$. Then, $\varphi\in\mathcal{A}(\C_+)$.
\end{proposition}
\begin{proof}
We want to show that the Dirichlet series defining $\varphi$ is in fact bounded in the whole right half-plane and by Bohr's theorem it also converges there. Since $\varphi\in\mathcal{D}$, there exists $\nu\in\R$ such that $\varphi$ is bounded in the closed half-plane $\overline{\C}_{\nu}$ by, say, $M>0$. We suppose that $\nu>0$, otherwise there is nothing to prove. By hypothesis, there exists $\delta>0$ such that $|\varphi(s_1)-\varphi(s_2)|<1$ for every $s_1,s_2\in \{s\in\C_+:\text{Re}(s)<\nu\}$ such that $|s_1-s_2|<\delta$. Take $k\in\N$ such that $\frac1k<\delta$. We consider $s\in\C_+$ with $\text{Re}(s)<\nu$. Then,
\begin{align*}
    |\varphi(s)|\leq\sum_{j=0}^{k-1}\left|
    \varphi\left(s+\left(\frac{j}{k}\right)\nu\right)-\varphi\left(s+\left(\frac{j+1}{k}\right)\nu\right)
    \right|
    +\left|\varphi\left(s+\nu\right)\right|\leq k+M.
\end{align*}
Then, the Dirichlet series $\varphi$ is bounded in $\C_+$ and, consequently, it converges there.
\end{proof}
\begin{remark}
    Notice that we have only used the uniform continuity of $\varphi$ in the vertical strip $\Omega=\{s\in\C_+:\text{Re}(s)\leq\sigma_u(\varphi)\}$, so we can ask $\varphi$ just to be a somewhere convergent Dirichlet series which is uniformly continuous in the strip $\Omega$.
\end{remark}
The algebra $\mathcal{A}(\C_+)$ is another example of how Dirichlet series behave differently to the analytic functions in the unit disc $\D$. Indeed, in the unit disc setting, the space of analytic functions which are uniformly continuous in $\D$ coincides with those analytic functions having a continuous extension to $\partial\D$. This is no longer true in the case of Dirichlet series.  Denote by $\mathcal{B}$ the class of functions in $\mathcal{H}^{\infty}$ having an extension to $\partial\C_+$. Trivially, $\mathcal{B}$ is also a closed Banach subalgebra of $\mathcal{H}^{\infty}$ containing $\mathcal{A}(\C_+)$. In \cite{pascal}, Lefèvre studied composition operators from $\mathcal{B}$ into $\mathcal{H}^{\infty}$. We see now that $\mathcal{A}(\C_+)\subsetneq\mathcal{B}$. 
\begin{proposition}\label{algebrab}
There exists a Dirichlet series belonging to $\mathcal{H}^{\infty}$ such that it has a continuous extension to $\partial\C_+$ but it is not uniformly continuous in $\C_+$.
\end{proposition}
\begin{proof}
We let $T:\D\to\C_+$ be the conformal mapping $T(z)=\frac{1+z}{1-z}$. Define the map $\varphi(s)=\frac122^{-s}-\frac123^{-s}$. We claim that 
\[
    \frac122^{-it}-\frac123^{-it}\not=1,\quad\text{for all $t\in\R$}.
    \]
Indeed, if there were a $t_0\in\R$ such that
    \[
    2^{-it_0}-3^{-it_0}=2,
    \]
then, we would have that $t_0=2\pi k/\log2$ and $t_0=\pi(2l+1)/\log3$, for $k,l\in\Z$. In particular, this would imply that $\log3/\log2$ is a rational number, which constitutes a contradiction. Therefore, there is no real number $t$ such that $\varphi(it)=1$. Now, by Kronecker's Theorem, there exists a sequence $\{t_n\}$ of real numbers such that $\varphi(it_n)\to1$ as $n\to\infty$. Let $\{h_n\}$ be a sequence of positive real numbers such that $h_n\to0$. Then, we set $s_n=h_n+it_n$. 
    \
    We claim the existence of a sequence $\{r_n\}$ in $\C_+$ such that $|r_n-s_n|\to0$ as $n\to\infty$ and, setting $z_n=\varphi(s_n)$, $w_n=\varphi(r_n)$, that
\[
i) \ z_n\to1,
\qquad
ii) \ \left|
\frac{T(z_n)}{T(w_n)}
\right|=e^{\pi}.
\]
Regarding $i)$, we know that $\varphi(it_n)\to1$. By the choice of $h_n$'s, we still have that $z_n=\varphi(h_n+it_n)\to1$. 

\ 
To prove $ii)$, first notice that since $z_n\to1$, we have that $|T(z_n)|\to\infty$. Then, take $\widetilde{s_n}=s_n+\varepsilon_n$, where $\{\varepsilon_n\}$ is a sequence of positive real numbers tending to zero such that for each $n$
\[
\frac12\left( 2^{-\varepsilon_n}+
3^{-\varepsilon_n}
\right)\leq 1-\frac{2e^{\pi}}{|T(z_n)|}.
\]
This gives $ii)$ since, clearly, $|u_n|=|\varphi(\widetilde{s_n})|\leq1/2(2^{-\varepsilon_n}+
3^{-\varepsilon_n})$. So,
\[
|u_n|\leq 1-\frac{2e^{\pi}}{|T(z_n)|}.
\]
This is,
\[
\frac{e^{\pi}}{|T(z_n)|}\leq\frac{1-|u_n|}{2}\leq\frac{|1-u_n|}{|1+u_n|}=\frac{1}{|T(u_n)|}.
\]
Reordering, we eventually get that
\[
\left|\frac{T(z_n)}{T(u_n)}
\right|\geq e^{\pi}.
\]
If we take the segments $\gamma_n(\alpha)=\alpha r_n+(1-\alpha)s_n$, $0\leq\alpha\leq1$, we can find $\alpha_n$ such that $|T(z_n)|/|T(w_n)|=e^\pi$, where $w_n=\varphi(\gamma_n(\alpha_n))$. Hence, claim $ii)$ follows with $r_n=\gamma_n(\alpha_n)$.

\ 
Let $A=\{z\in\C: e^{-\pi/2}<|z|<e^{\pi/2}\}$. Then,  the holomorphic function $f(z)=\exp(i\log T(z))$ maps $\D$ into $A$. Here $\log$ stands for the complex logarithm where the argument is taken in $(-\pi/2,\pi/2)$. The Dirichlet series $F=f\circ\varphi$ is clearly convergent in $\C_+$ and it satisfies the thesis of the Proposition. Indeed, since $\varphi$ can be clearly 
continuously extended to $\partial\C_+$, $1\not\in\varphi(\partial\C_+)$, and $f$ is continuous in $\overline{\D}\setminus\{1\}$, we conclude that $F$ can be continuously extended to $\partial\C_+$. Regarding the failure of the uniform continuity on $\C_+$, consider the sequences $\{s_n\}$ and $\{r_n\}$ from above. By construction, $|s_n-r_n|\to0$ as $n\to\infty$. Now, thanks to part $ii)$ of the claim, we have that
\[
|\text{Re}(\log(T(z_n)))-\text{Re}(\log(T(w_n)))|=\pi.
\]
Hence, 
\[
|\text{Im}(i\log(T(z_n)))-\text{Im}(i\log(T(w_n)))|=\pi.
\]
Recalling the definition of $f$, we find that
\[
|\text{Arg}(f(z_n))-\text{Arg}(f(w_n))|=\pi.  
\]
Hence, $f(z_n)$ and $f(w_n)$ lie in antipodal segments joining the inner circle and the outer circle of the boundary of $A$. Therefore, between these two points we can fit, at least, a ball of radius $e^{-\pi/2}$. Therefore, $|f(z_n)-f(w_n)|\geq 2e^{-\pi/2}$, and the conclusion follows.
\end{proof}
\begin{remark}
    The algebra $\mathcal{B}$ is not separable. We can see this using the ideas from the proof of the latter proposition. Let $$M=\left\{2^{-it}:t=\frac{(2k+1)\pi}{\log3-\log2}, \ k\in\Z\right\}.$$ For $\tau\in\partial\D$, define $f_{\tau}(z)=\exp(i\log T(\overline{\tau }z))$, $z\in\D$. Then, we consider $F_{\tau}=f_{\tau}\circ\varphi$, where $\varphi(s)=\frac122^{-s}-\frac123^{-s}$. We claim the following:
   \begin{itemize}
       \item[i)] If $\tau\in \partial\D\setminus M$, then $F_{\tau}\in\mathcal{B}$.
       \item[ii)] For $\tau,\tau'\in\partial\D\setminus M$, $\tau\not=\tau'$:
       \[
       \|F_{\tau}-F_{\tau'}\|_{\infty}\geq e^{-\pi/2}.
       \]
   \end{itemize}
    We begin by showing i). The function $f_{\tau}$ fails to be continuous whenever $\overline{\tau}z=1$. Hence, so does $F_{\tau}$  for those $t\in\R$ such that $\varphi(it)=\tau$. If this occurs, we have that,
    \[
    \frac122^{-it}-\frac123^{-it}=\tau.
    \]
    This is, the middle-point $\tau$ of the points $2^{-it}$ and $-3^{-it}$ in $\partial\D$ lies also in $\partial\D$. This implies that $2^{-it}=-3^{-it}=\tau$. Or, equivalently, $2^{it}3^{-it}=-1$. This forces $t=\frac{(2k+1)\pi}{\log3-\log2}$, $k\in\Z$, and the claim follows.

    \ 
    Now, we prove ii). Fix $\tau,\tau'\in\partial\D\setminus M$, $\tau\not=\tau'$. By Kronecker's Theorem (see \cite[Proposition 3.4]{peris}), there exists a sequence $\{t_n\}$ of real numbers such that $\varphi(it_n)\to\tau$ as $n\to\infty$. Let $\{h_n\}$ be a sequence of positive real numbers such that $h_n\to0$. We set $s_n=h_n+it_n$. Mimicking the construction from the proof of Proposition \ref{algebrab}, we can find a sequence $\{r_n\}\in\C_+$, $\varphi(r_n)\to\tau$, such that 
    \begin{equation}\label{acoinf}
        |F_{\tau}(s_n)-F_{\tau}(r_n)|\geq 2e^{-\pi/2}.
    \end{equation}
    Since $\tau\not=\tau'$, $f_{\tau'}$ can be continuously extended to $\overline{\D}\setminus\{\tau'\}$ and $\varphi(s_n),\varphi(r_n)\to\tau$, as $n\to\infty$, we have that
    \begin{equation}\label{tiendecero}
        |F_{\tau'}(s_n)-F_{\tau'}(r_n)|\to0.
    \end{equation}
    Using both \eqref{acoinf} and \eqref{tiendecero}, 
    \begin{align*}
        2e^{-\pi/2}\leq |F_{\tau}(s_n)-F_{\tau}(r_n)|\leq 2\|F_{\tau'}-F_{\tau}\|_{\infty}+ |F_{\tau'}(s_n)-F_{\tau'}(r_n)|
    \end{align*}
    for all $n\in\N$. Letting $n\to\infty$, the claim ii) follows.

    \ 
    We have shown that the are uncountable many disjoint open sets in $\mathcal{B}$, giving the non-separability of the algebra $\mathcal{B}$.
\end{remark}

\subsection{Composition operators}
In this section, we give a necessary and sufficient condition for the boundedness of the composition operators $C_{\Phi}:\mathcal{A}(\C_+)\to\mathcal{A}(\C_+)$. This condition will be given in terms of some properties satisfied by the symbols $\Phi$. In the seminal work \cite{gorheda}, Gordon and Hedenmalm characterised the symbols giving rise to bounded composition operators in the Hilbert Dirichlet space $\mathcal{H}^2$. To do so, they introduced the now so-called \emph{Gordon-Hedenmalm class} (see Definition \ref{def: gorhed}). Gordon and Hedenmalm proved the following characterisation.
\begin{theorem}[Gordon-Hedenmalm]\label{gordon}
An analytic function $\Phi:\C_{\frac12}\to\C_{\frac12}$ defines a bounded composition operator $\mathcal{C}_{\Phi}:\mathcal{H}^2\to\mathcal{H}^2$ if and only if $\Phi$ has a holomorphic extension to $\C_{+}$ that  belongs to the class $\mathcal{G}$. 
\end{theorem}
The characterisation of the boundedness of composition operators on $\mathcal{H}^\infty$  was given by Bayart in \cite{bayarto} (see also \cite[Proposition 2]{bayart-castillo}).
\begin{theorem}[Bayart]\label{boundedness-Hinfty}
A function $\Phi:\C_{+}\to\C_{+}$ defines a bounded composition operator $\mathcal{C}_{\Phi}:\mathcal{H}^\infty\to\mathcal{H}^\infty$ if and only if $\Phi$ belongs to the class $\mathcal G_{\infty}$.
\end{theorem}

The condition $\Phi$ belonging to $\mathcal{G}_{\infty}$ and being uniformly continuous in $\C_+$ is sufficient to ensure the boundedness of the operator $C_{\Phi}$ on $\mathcal{A}(\C_+)$. However, it is no longer necessary (see Example \ref{ex: gea no unifcon}). 
The boundedness of $C_{\Phi}$ on $\mathcal{A}(\C_+)$ is characterised by the membership of $\Phi$ to the class $\mathcal{G}_{\mathcal{A}}$, see Definition \ref{def: ga}, where uniform continuity is only required in the sets 
\[
A_M:=\{s\in\C_+:0<\text{Re}(\Phi(s))<M\},\quad\text{for all $M>0$}.
\]

We now prove the main result of the section.
\begin{theorem}\label{thsimal}
Let $\Phi:\C_+\to\C_+$ be analytic. Then, the following statements are equivalent:
\begin{enumerate}[a)]
    \item $\Phi$ defines a bounded composition operator $C_{\Phi}$ on $\mathcal{A}(\C_+)$.
    \item $\Phi\in\mathcal{G}_{\mathcal{A}}$.
    \item $n^{-\Phi}\in{\mathcal{A}}(\C_+)$ for all $n\in\N$.
    \item $\Phi\in\mathcal{G}_{\infty}$ and $n^{-\Phi}\in \mathcal{A}(\C_+)$ for all $n\in\N$.
    \item $\Phi\in\mathcal{G}_{\infty}$ and there exists some $n\geq2$ such that $n^{-\Phi}\in \mathcal{A}(\C_+)$.
\end{enumerate}
\end{theorem}
\begin{remark}\label{rem: constant}
Notice that in the statement of the theorem we are excluding the case in which the symbol $\Phi$ touches the boundary of $\C_+$, this is, $\Phi(\C_+)\cap \partial\C_+\not=\emptyset$. If this happens, then the symbol $\Phi$ is constant and the composition operator $C_{\Phi}$ is trivially bounded on $\mathcal{A}(\C_+)$. Observe also that this is not the case in $\mathcal{H}^{\infty}$ since not every $\mathcal{H}^{\infty}$ function is defined in the closure of the right half-plane.
\end{remark}
\begin{proof}[Proof of Theorem \ref{thsimal}]
Trivially, $d)$ implies $e)$. Now, since $n^{-s}\in\mathcal{A}(\C_+)$ for all $n\in\N$, we clearly have that $a)$ implies $c)$.

\ 
To show $c)$ implies $d)$, it is enough to prove that $\Phi\in\mathcal{G}_{\infty}$. To do so, we use Theorem \ref{lemma: bay-cas} and the conclusion immediately follows.

\ 
Let us show that $d)$ implies $a)$. By Theorem \ref{boundedness-Hinfty}, $C_{\Phi}$ maps $\mathcal{H}^{\infty}$ into itself because $\Phi\in\mathcal{G}_{\infty}$. Moreover, we have by hypothesis that $C_{\Phi}$ sends the Dirichlet monomials $n^{-s}$ into $\mathcal{A}(\C_+)$. Since the algebra is the closure of the linear span of the monomials in the $\mathcal{H}^{\infty}$ norm (Theorem \ref{thpropbasa} (2)) and the operator $C_{\Phi}$ is linear, we conclude that $C_{\Phi}$ actually maps boundedly the algebra of Dirichlet series $\mathcal{A}(\C_+)$ into itself.

\ 
For $b)$ implies $c)$, consider $n\in\N$, $n\geq2$, since for $n=1$ there is nothing to prove. We want to show that $g(s)=n^{-\Phi(s)}$ is uniformly continuous in $\C_+$. To do so, take $\varepsilon>0$ and $M>0$ so that $n^{-(M-1)}<\varepsilon/4$. 
By the uniform continuity of $n^{-s}$ in $\C_+$, there exists $\delta_1>0$ such that $|n^{-z}-n^{-w}|<\varepsilon/2$  whenever $|z-w|<\delta_1$.
By the definition of $\mathcal{G}_{\mathcal{A}}$, we have that the symbol $\Phi$ is uniformly continuous in $A_M$. Then, there exists a $\delta>0$ so that whenever $s_1,s_2$ belong to $A_M$ and $|s_1-s_2|<\delta$, we have $|\Phi(s_1)-\Phi(s_2)|<\delta_{1}$ and thus $|n^{-\Phi(s_{1})}-n^{-\Phi(s_{2})}|<\varepsilon/2$. We may assume that $\delta<\delta_{1}$. 

\ 
Now, suppose that neither $s_1$ nor $s_2$ belong to $A_{M-1}$. 
Then, $\text{Re}(\Phi(s_j))>M-1$, $j=1,2$. Therefore, by the choice of $M>0$
\[
|n^{-\Phi(s_1)}-n^{-\Phi(s_2)}|\leq n^{-\text{Re}(\Phi(s_1))}+n^{-\text{Re}(\Phi(s_2))}\leq 2n^{-(M-1)}<\varepsilon/2.
\]
Eventually, suppose that $s_1\in A_{M-1}$ and $s_2\not\in A_M$ and $|s_1-s_2|<\delta$. 
 By the mean value theorem, if $\gamma(t)=ts_1+(1-t)s_2$, $0\leq t\leq1$, there exists $t_0$ so that 
\[
M-1<\text{Re}(\Phi(\gamma(t_0))<M.
\]
Set $s_3=\gamma(t_0)$. Clearly, $s_1,s_3\in A_M$ and $|s_1-s_3|<\delta$. Therefore, $|n^{-\Phi(s_1)}-n^{-\Phi(s_3)}|<\varepsilon/2.$ In addition, $s_2,s_3\notin A_{M-1}$ and  $|n^{-\Phi(s_2)}-n^{-\Phi(s_3)}|<\varepsilon/2$.
Hence, 
\[
|n^{-\Phi(s_1)}-n^{-\Phi(s_2)}|\leq |n^{-\Phi(s_1)}-n^{-\Phi(s_3)}|+
|n^{-\Phi(s_3)}-n^{-\Phi(s_2)}|<\frac{\varepsilon}{2}+\frac{\varepsilon}{2}=\varepsilon.
\]

\  
We conclude the proof showing $e)$ implies $b)$. By hypothesis, there exists an $n\geq2$ such that the Dirichlet series $n^{-\Phi}$ belongs to $\mathcal{A}(\C_+)$. Assume that $\Phi\not\in\mathcal{G}_{\mathcal{A}}$. Since $\Phi\in\mathcal{G}_{\infty}$, there must exist $M_0>1$, $\varepsilon>0$, $\{r_k\}$, $\{t_k\}\subset A_{M_0}$ such that $|r_k-t_k|<\frac 1k$ for all $k\in\N$ and
\[
|\Phi(r_k)-\Phi(t_k)|>\varepsilon, \quad \text{for all  } k\in\N.
\]
Consider the rectangle $\mathcal{R}=\{z\in\C_+: \ \text{Re}(z)\leq2M_0, \ |\text{Im}(z)|\leq \pi/\log(n)\}$. The function $f(s)=n^{-s}$ is bilipschitz there, so, for some constant $K=K(M_0)$
\[
|n^{-s}-n^{-t}|\geq K|s-t|, \quad s,t\in R.
\]
Let $a_k=\frac12(\text{Im}(\Phi(r_k))+\text{Im}(\Phi(t_k)))$. Hence, if for infinitely many $k\in\N$, the points $\Phi(r_k)-ia_k$ and $\Phi(t_k)-ia_k$ belong to the rectangle $\mathcal{R}$, we would have that
\[
|n^{-\Phi(r_k)}-n^{-\Phi(t_k)}|=|n^{-(\Phi(r_k)-ia_k)}-n^{-(\Phi(t_k)-ia_k)}|\geq K\varepsilon.
\]
Therefore, $n^{-\Phi(s)}$ would not be uniformly continuous in $\C_+$. Then, we can assume that $\text{Im}(\Phi(t_k)-\Phi(r_k))>\pi/\log n$ for every $k$. For each $k\in\N$, we consider the segments  $\gamma_k(\alpha)=\alpha r_k+(1-\alpha)t_k$, $0\leq \alpha\leq1$, joining $r_k$ and $t_k$. We now take the image of $\gamma_k$ under $\Phi$ and one of the following two cases must hold: 

\noindent {\sl Case I.} First, assume, for infinitely many $k$, the existence of an $\alpha_k\in [0,1]$ such that $\gamma_k(\alpha_k)\not\in A_{2M_0}$. Of course, we may assume that this happens for all $k$. Set $u_k=\gamma_k(\alpha_k)$ for which, clearly, $|r_k-u_k|<1/k$ for all $k\in\N$. Then,
$|n^{-\Phi(u_k)}|<n^{-2M_0}$ and $|n^{-\Phi(r_k)}|>n^{-M_0}$ (since $\{r_k\}$ is in $A_{M_0}$). Therefore,
\[
|n^{-\Phi(u_k)}-n^{-\Phi(r_k)}|>n^{-M_0}.
\]
This clearly implies that $n^{-\Phi(s)}$ is not uniformly continuous in $\C_+$, a contradiction.

\noindent {\sl Case II.}  Assume now that the whole image of $\gamma_k$ under $\Phi$ lies in $A_{2M_0}$ for infinitely many $k$. Again we may assume that this happens for all $k$.  If this were the case, for each $k$ we could find an $\alpha_k\in [0,1]$ so that 
\[
\text{Im}(\Phi(\gamma_k(\alpha_k))-\Phi(r_k))=\pi/\log n.
\]
Set again $u_k=\gamma(\alpha_k)$. Hence,
\[
\text{Arg}(n^{-\Phi(u_k)})-\text{Arg}(n^{-\Phi(r_k)})=\pi,\quad \text{for all $k$}.
\]
Since the image of the vertical strip $0<\text{Re}(s)<2M_0$ under the function $f(s)=n^{-s}$ is the annulus
$
\{z\in\D: |z|>n^{-2M_0}\},
$
we deduce that the modulus of the difference of two points whose argument differs exactly on $\pi$ units will be strictly greater than $2n^{-2M_0}$. Hence, by the choice we made of both $u_k$ and $r_k$, we have that
\[
|n^{-\Phi(u_k)}-n^{-\Phi(r_k)}|>n^{-2M_0},\quad \text{for all $k$}.
\]
Again we have that this implies that $n^{-\Phi(s)}$ is not uniformly continuous in $\C_+$, a contradiction.
\end{proof}


\begin{example}\label{example1}\emph{The classes $\mathcal G_{\mathcal A}$ and $\mathcal G_\infty$ do not coincide.} In order to prove this statement, we need to build a symbol $\Phi$ in $\mathcal{G}_{\infty}$ such that, for some $M_0>0$, $\Phi$ is not uniformly continuous in $A_{M_0}$. As it was seen in Definition \ref{def: gorhed}, the symbols in the class $\mathcal{G}_{\infty}$ consist on the analytic functions from $\C_+$ to $\C_+$ which may be written as $c_{\Phi}s+\varphi(s)$, where $c_{\Phi}\in\N\cup\{0\}$ and $\varphi\in\mathcal{D}$. We are going to construct a symbol $\Phi$ with $c_{\Phi}=0$. An easy way to obtain a Dirichlet series is through the composition of a holomorphic function on $\D$ with a function $g(s)=n^{-s}$, $n\in\N$, $n\geq2$, defined on $\C_+$. We take $n=2$. Therefore, we consider $f\in H^{\infty}(\D)\setminus A(\D)$, where $A(\D)$ stands for the disc algebra, this is, the set of holomorphic functions in $\D$ which can be continuously extended to $\overline{\D}$. Note that $f$ cannot be uniformly continuous in $\D$ since, if it were the case, then we could extend it continuously to $\overline{\D}$ and it would belong to $A(\D)$. Now, the function $\Phi(s)=f(2^{-s})+\|f\|_{\infty}$ is analytic on $\C_+$ and it defines a Dirichlet series there. Being clear that $f$ maps $\C_+$ into $\C_+$ it remains to check that the resulting function is not uniformly continuous in some $A_{M_0}$, $M_0>0$. 
Clearly, $\Phi$ is not uniformly continuous on $\C_+$. Notice that $\Phi(\C_+)$ is bounded. So, for $M>0$ big enough, $A_M=\C_+$. Hence, for those $M$'s, $\Phi$ cannot be uniformly continuous in $A_M$, as desired. In brief, the symbol $\Phi$ induces a bounded composition operator in $\mathcal{H}^{\infty}$ but not in $\mathcal{A}(\C_+)$.
\end{example}

\begin{example}\label{ex: gea no unifcon}
\emph{ There are non-uniformly continuous symbols in $\C_+$ which belong to $\mathcal{G}_{\mathcal{A}}$.} We now give an example of how the uniform continuity of the symbol in $\C_+$ is not a necessary condition in Theorem \ref{thsimal}. To do so, we are going to build  $\Phi\in\mathcal{G}_{\mathcal{A}}$ which is not uniformly continuous in $\C_+$. Let us consider $\Phi(s)=T(2^{-s})$, where $T$ is a conformal map sending the unit disc onto the intersection of $\C_+$ with a horizontal strip. In addition, assume that $T(1)=\infty$. Clearly, $\Phi:\C_+\to\C_+$ is analytic and it also defines a Dirichlet series there. Therefore, $\Phi\in \mathcal{G}_{\infty}$. Now, for every $M>0$, $T^{-1}(A_M)$ is contained in the compact set $\overline{\D}\setminus E$, where $E$ is an open neighbourhood of $1$ depending on $M$. Hence, $T$ is uniformly continuous there. Since $f(s)=2^{-s}$ is uniformly continuous in the whole $\C_+$, we conclude that $\Phi$ is in the class $\mathcal{G}_{\mathcal{A}}$. Nonetheless, notice that $\Phi$ cannot be uniformly continuous in $\C_+$ since it maps bounded neighbourhoods of the points $s_n=2\pi in/\log 2$, $n\in\Z$, to unbounded regions.
\end{example}

\section{The compactness matter}\label{sec:compactness}
This section is devoted to compactness and weak compactness of composition operators in the algebra of Dirichlet series $\mathcal{A}(\C_+)$. Both in $H^{\infty}(\D)$ and in the disc algebra $A(\D)$, the class of compact composition operators and the class of weakly compact composition operators coincide and can be characterised in terms of the range of the symbol. Namely, if $\phi$ is such a symbol, the distance of the set $\phi(\D)$ to $\partial \D$ must be positive. In addition, in both spaces the non-compactness is due to the fact that the operator acts  as an isomorphism in a non-reflexive subspace, isomorphic to $c_{0}$ in the disc algebra and  to $\ell^{\infty}$ in the case of $H^{\infty}(\D)$ (see the survey \cite{Contreras-Diaz} to see all these results). In this section we will show that  the same results hold in our setting.
In the case of $\mathcal{H}^{\infty}$, something is known. Namely, Bayart \cite{bayarto} proved that a bounded composition operator $C_{\Phi}$ is compact on $\mathcal{H}^{\infty}$ if and only if $\inf \{\Re \, \Phi(s):\, s\in \C_{+}\}>0$ and Lefèvre \cite{pascal} showed that compactness coincides with weak compactness for composition operators on $\mathcal{H}^{\infty}$. This topic was treated in a different context in \cite[Proposition 3]{aron-ga-lind}.

Given an interpolating sequence $\{z_{n}\}$ in the unit disc, 
 by Beurling's Theorem \cite[Theorem 2.1, page 285]{Garnett}, there are a constant $M>0$ and a sequence of functions $\{g_{j}\}$ in  $H^{\infty}(\D)$ such that
\begin{equation}\label{Eq:fixedcopy1}
g_{j}(z_{j})=1, \quad \quad g_{j}(z_{k})=0,\quad \text{ if  } j\neq k,
\end{equation}
and
\begin{equation}\label{Eq:fixedcopy2}
\sum_{j=1}^{\infty} |g_{j}(z)|\leq M, \quad \text{ for all } z\in \D.
\end{equation}
In general, it is not possible to get the above functions $g_{j}$ belonging to the disc algebra. Next lemma shows that if the sequence $\{z_{n}\}$ converges to a point of the boundary, passing to a subsequence $\{z_{n_j}\}$, then the functions can be chosen continuous up to the boundary of the unit disc:

\begin{lemma}\label{Lem:disc-algebra}
Let $\{z_{n}\}$ be a sequence in the unit disc converging to a point $\tau\in \partial \D$. Then, there exist a subsequence $\{z_{n_{j}}\}$, a sequence of functions $\{g_{j}\}$ in the disc algebra $A(\D)$, and a constant $M$ such that
\begin{equation}\label{Eq:disc-algebra1}
g_{j}(z_{n_{j}})=1, \quad  \quad g_{j}(z_{n_{k}})=0,\quad \text{ if  } j\neq k,
\end{equation}
and
\begin{equation}\label{Eq:disc-algebra2}
\sum_{j=1}^{\infty} |g_{j}(z)|\leq M, \quad \text{ for all } z\in \D.
\end{equation}
\end{lemma}
\begin{proof} To simplify the exposition, we assume that $\tau=1$. Take $\Omega=D(-1,2)$,  the disc centred at $-1$ and radius $2$. Notice that the points $z_{n}$ belong to $\Omega$ and the sequence $\{z_{n}\}$ converges to a point of the boundary of $\Omega$.  
By  \cite[Theorem 1.1, page 278]{Garnett}, there exists a subsequence $\{z_{n_{j}}\}$ which is interpolating for $H^{\infty}(\Omega)$.
Thus, by Beurling's Theorem \cite[Theorem 2.1, page 285]{Garnett}, we can find a constant $M_{1}>0$ and a sequence of functions $\{h_{j}\}$ in  $H^{\infty}(\Omega)$ such that
\begin{equation}\label{Eq:fixedcopy3}
h_{j}(z_{n_{j}})=1, \quad \quad h_{j}(z_{n_{k}})=0,\quad \text{ if  } j\neq k,
\end{equation}
and
\begin{equation}\label{Eq:fixedcopy4}
\sum_{j=1}^{\infty} |h_{j}(z)|\leq M_1, \quad \text{ for all } z\in \Omega.
\end{equation}
Notice that, for each $j$, $h_{j}$ is bounded, analytic in the unit disc and continuous in $\overline\D\setminus \{1\}$. Now, for each $j$, take $T_{j}$ an automorphism of the unit disc such that $T_{j}(z_{n_{j}})=0$ and $T_{j}(1)=1$. Finally, consider $g_{j} (z)=(1-T_{j}(z))h_{j}(z)$, $z\in \D$. The boundedness of $h_{j}$ implies that $\lim_{z\to 1} g_{j}(z)=0$ so that $g_{j}\in A(\D)$ for all $j$. Using \eqref{Eq:fixedcopy3} and \eqref{Eq:fixedcopy4}, a straightforward computation shows that these functions satisfies \eqref{Eq:disc-algebra1} and \eqref{Eq:disc-algebra2} with $M=2M_{1}$. 
\end{proof}

Let us recall that given three Banach spaces $X$, $Y$, and $E$, an operator $T : X \to Y$ {\sl is said to fix a copy of $E$} if there is a subspace $Z\subseteq X$ such that $Z$ is isomorphic to $E$ and $T : Z \to T(Z)$ is an isomorphism. In our case, $E$ will be either the sequence space $\ell^{\infty}$ or $c_0$.

\begin{theorem}\label{Thm:Hinfty}
Let $\Phi\in \mathcal G_{\infty}$ such that $\inf \{\Re \, \Phi(s):\, s\in \C_{+}\}=0$. Then
\begin{itemize}
\item[1)] $C_{\Phi}:\mathcal H^{\infty}\to \mathcal H^{\infty} $ fixes a copy of $\ell ^{\infty}$; 
\item[2)] $C_{\Phi}: \mathcal A (\C_{+})\to \mathcal H^{\infty} $ fixes a copy of $c_{0}$.
\end{itemize}
If, in addition,  $\Phi\in \mathcal G_A$, then $C_{\Phi}: \mathcal A (\C_{+})\to \mathcal A (\C_{+}) $ fixes a copy of $c_{0}$.
\end{theorem}
\begin{proof}
We begin proving $2)$. By hypothesis, there is a sequence $\{w_{j}\}$ in $\C_{+}$ such that $\Re \, w_{j}$ goes to zero as $j$ goes to $+\infty$, where $w_{j}=\Phi(s_{j})$ for some $s_{j}\in \C_{+}$. Passing to a subsequence if necessary, we may assume that $\{2^{-w_{j}}\}$ converges to $\tau\in\partial\D$. By Lemma \ref{Lem:disc-algebra}, there exists a subsequence $\{w_{n_{j}}\}$, a sequence of functions $\{g_{j}\}$ in the disc algebra $A(\D)$, and a constant $M$ such that
\begin{equation}\label{Eq:disc-algebra7}
g_{j}(2^{-w_{n_{j}}})=1, \quad  \quad g_{j}(2^{-w_{n_{k}}})=0,\quad \text{ if  } j\neq k,
\end{equation}
and
\begin{equation}\label{Eq:disc-algebra8}
\sum_{j=1}^{\infty} |g_{j}(z)|\leq M, \quad \text{ for all } z\in \D.
\end{equation}
Write $f_{j}(s)=g_{j}(2^{-s})$ for all $s\in \C_{+}$ and for all $j$. The functions $f_{j}$ belong to $\mathcal A(\C_{+})$ for all $j$.
Take $Z$ the closed subspace in $\mathcal A(\C_{+})$ generated by $\{f_{j}: j\geq 1\}$ and take $T:c_{0}\to Z$ given by $T(\{\alpha_{j}\})=\sum _{j=1}^{\infty }\alpha_{j}f_{j}$, for all $\{\alpha_{j}\}\in c_{0}$. By \eqref{Eq:disc-algebra8}, $T$ is well-defined. Moreover, $T$ is an isomorphism. Indeed, by \eqref{Eq:disc-algebra8}, for all $j$,
$$
|| T (\{ \alpha_{j}\}) ||\leq \sup_{w\in \C_{+}} |\sum_{j}\alpha_{j}f_{j}(w)|\leq ||\{ \alpha_{j}\}||_{\infty} M
$$ 
and, by \eqref{Eq:disc-algebra7},
\begin{equation}\label{Eq:disc-algebra9}
|\alpha_{j}|=|\sum_{k}\alpha_{k}g_{k}(2^{-w_{n_{j}}})|= |\sum_{k}\alpha_{k}f_{k}(w_{n_{j}})|= |T(\{\alpha_{k}\})(w_{n_{j}})|
\leq \|T(\{\alpha_{k}\})\|.
\end{equation}
That is,
$$
\|\{\alpha_{k}\}\|\leq \|T(\{\alpha_{k}\})\| \leq M\|\{\alpha_{k}\}\|
$$
for all $\{\alpha_{k} \} \in c_{0}$.
Moreover, given $f=T(\{\alpha_{k}\})\in Z$,  by \eqref{Eq:disc-algebra9},
$$
\| C_{\Phi}(f)\|\geq \sup_{j}|f(\Phi(s_{n_{j}}))|=\sup_{j}|f(w_{n_{j}})|\geq \sup_{j}|\alpha_{j}|=\|\{ \alpha_{j}\}\|\geq \|f\|/M.
$$
That is, ${C_{\Phi}}_{|Z}$ is an isomorphism. This concludes the proof of $2)$.

To prove $1)$, consider the operator $S:\ell^{\infty }\to \mathcal H^{\infty}$ given by $S(\{\alpha_{j}\})=\sum _{j=1}^{\infty }\alpha_{j}f_{j}$, for all $\{\alpha_{j}\}\in \ell^{\infty}$. The great difference between  $S$ and $T$ is that now the sequence $\{\alpha_{j}\}$ does not go to zero and, thus, it is not clear that $S(\{\alpha_{j}\})$ is well-defined. By \eqref{Eq:disc-algebra8}, what we know is that, for each $s$, the series $S(\{\alpha_{j}\})(s)=\sum _{j=1}^{\infty }\alpha_{j}f_{j}(s)$ converges, so that   $S(\{\alpha_{j}\})(s)$ is a well-defined function bounded by $M\|\{\alpha_{j}\}\|$. We claim that $S(\{\alpha_{j}\})$ belongs to 
$\mathcal H^{\infty}$.
To prove the claim, notice that the sequence of functions  $\{h_{N}\}=\{\sum _{j=1}^{N}\alpha_{j}f_{j}\}$ is bounded (by $M\|\{\alpha_{j}\}\|$) in $\mathcal H^{\infty}$.  By Bayart-Montel's Theorem (Theorem \ref{th: bayart}), it has a subsequence that converges uniformly in $\C_{\varepsilon}$, $\varepsilon>0$, to a certain $h\in\mathcal H^{\infty}$. But $\{h_{N}\}$ converges pointwise to $S(\{\alpha_{j}\})$. Thus, $h=S(\{\alpha_{j}\})$ and the claim is proven.
The remaining properties of $S$ can be obtained in a similar way to those of $T$ and we are done.
\end{proof}
Once the above theorem has been proven, we are ready to state and prove the characterisation of compact composition operators in the Dirichlet series algebra $\mathcal{A}(\C_+)$. Clearly, if $\Phi$ is as in Remark \ref{rem: constant}, the operator $C_{\Phi}$ is compact. Because of this, in the following theorem we only consider symbols satisfying $\Phi(\C_+)\subset\C_+$.
\begin{theorem}\label{thcomp}
Let $C_{\Phi}:\mathcal{A}(\C_+)\to \mathcal{A}(\C_+) $ be a bounded composition operator with a symbol $\Phi:\C_{+}\to \C_{+}$. Then, the following statements are equivalent:
\begin{enumerate}[a)]
    \item $\Phi(\C_+)\subset\C_{\varepsilon}$, for some $\varepsilon>0$.
    \item $C_{\Phi}$ is compact.
    \item $C_{\Phi}$ is weakly compact.
    \item $C_{\Phi}:\mathcal{A}(\C_+)\to \mathcal{A}(\C_+) $  does not fix a copy of $c_{0}$.
\end{enumerate}
\end{theorem}
\begin{proof} By Theorem \ref{Thm:Hinfty}, we have that $d)$ implies $a)$. Being clear that $b)$ implies $c)$  and that $c)$ implies $d)$, it remains to prove that $a)$ implies $b)$. This follows directly from the argument in \cite[Theorem 18]{bayarto}. We include the proof for the sake of completeness. Therefore, assume the existence of $\varepsilon_0>0$ so that $\Phi(\C_+)\subset\C_{\varepsilon_0}$. Consider $\{f_n\}_{n\in\N}$ a bounded sequence in $\mathcal{A}(\C_+)$. Since the algebra of Dirichlet series is a closed subspace of $\mathcal{H}^{\infty}$, we have that the sequence $\{f_n\}_{n\in\N}$ is bounded there. Hence, by Bayart-Montel's Theorem (Theorem \ref{th: bayart}), we know the existence of both a convergent subsequence $\{f_{n_k}\}_{k\in\N}$ and a limit function $f\in\mathcal{H}^{\infty}$ so that $f_{n_k}\to f$ uniformly in half-planes $\C_{\varepsilon}$ for every $\varepsilon>0$. We take $\varepsilon=\varepsilon_0$. Therefore, by the hypothesis, $f_{n_k}\circ\Phi\to f\circ\Phi$ uniformly in $\C_+$. Since $\Phi\in\mathcal{G}_{A}$, by Theorem \ref{thsimal}, $f_{n_k}\circ\Phi\in\mathcal{A}(\C_+)$ for all $k$. Since $\mathcal{A}(\C_+)$ is closed, we conclude that $f\circ\Phi\in\mathcal{A}(\C_+)$. This yields the compactness of the operator $C_{\Phi}$. 
\end{proof}
The same proof given above for the algebra  of Dirichlet series works for $\mathcal{H}^{\infty}$ showing that $C_{\Phi}$ is compact on $\mathcal{H}^{\infty}$ if and only if it does not fix a copy of $\ell^{\infty}$. Let us recall that the equivalence between $a)$ and $b)$ was previously proved by Bayart in \cite{bayarto} and the equivalence between $b)$ and $c)$ was proved by Lefèvre in \cite{pascal}.

\begin{theorem}\label{thcompa}
Let $C_{\Phi}\colon\mathcal H^{\infty}\to \mathcal H^{\infty} $ be a bounded composition operator. Then, the following statements are equivalent:
\begin{enumerate}[a)]
    \item $\Phi(\C_+)\subset\C_{\varepsilon}$, for some $\varepsilon>0$.
    \item $C_{\Phi}$ is compact.
    \item $C_{\Phi}$ is weakly compact.
    \item $C_{\Phi}:\mathcal H^{\infty}\to \mathcal H^{\infty} $  does not fix a copy of $\ell^{\infty}$.
\end{enumerate}
\end{theorem}
Since the space $\mathcal{A}(\C_+)$ is separable, we obtain the following consequence.
\begin{corollary}
    Let $C_{\Phi}\colon \mathcal{H}^{\infty}\to\mathcal{A}(\C_+)$ be a bounded composition operator. Then, $C_{\Phi}$ is compact.
\end{corollary}

\section{About the closure of Gordon-Hedenmalm classes}\label{sec:closure-classes}
This section is devoted to the study of the closure of the classes of symbols $\mathcal{G}$,  $\mathcal{G}_{\infty}$ and $\mathcal{G_{\mathcal{A}}}$ under the uniform convergence on compacta. A first result related to the local uniform limit of Dirichlet series was proven in \cite[Theorem 4.2]{Aron-Bayart-cia}:
\begin{theorem}
Let $\Omega\subset\C$ be an open simply connected set and $H(\Omega)$ be the space of holomorphic functions on $\Omega$ endowed with the topology of uniform convergence on compacta. Then, Dirichlet polynomials are dense in $H(\Omega)$.
\end{theorem}
In particular, this theorem implies that, in general, the local uniform limit of Dirichlet polynomials need not be a Dirichlet series. However, we shall show in Theorem \ref{closeness}  that if, in addition, we assume that the Dirichlet polynomials map the right half-plane into itself, then the limit is indeed a Dirichlet series. 

\
We shall give two positive results, namely, the classes $\mathcal{G}$ and $\mathcal{G}_{\infty}$ are `almost' closed under the local uniform convergence and a negative one: the class $\mathcal{G_{\mathcal{A}}}$ fails to satisfy this property. Furthermore, $\overline{\mathcal{G}_{\mathcal{A}}}=\overline{\mathcal{G}_{\infty}}$.

\ 

\begin{theorem}\label{closeness}
Let $\{\Phi_n\}_n$ be a sequence of functions in $\mathcal{G}_{\infty}$ converging uniformly on compact sets of $\C_+$. Then, $\Phi_n\to\Phi$ uniformly in half-planes $\C_{\varepsilon}$, $\varepsilon>0$, and either
\begin{itemize}
\item[(1)] $\Phi\in\mathcal{G}_{\infty}$, or
\item[(2)] $\Phi(s)= it$, for some $t\in\R$.
\end{itemize}
 In other words, the class $\mathcal{G}_{\infty}\cup\{i\R\}$ is closed when it is endowed with the uniform convergence on compacta. In case $(1)$, we also have that $c_{\Phi_n}= c_{\Phi}$ for $n$ large enough. In case $(2)$, we have that $c_{\Phi_n}=0$ for $n$ large enough.
\end{theorem}
\begin{proof}
Let $\{\Phi_n\}_n$ be a sequence of elements in the class $\mathcal{G}_{\infty}$ converging uniformly on compact sets of $\C_+$ to the function $\Phi$. Clearly, $\Phi:\C_+\to\overline{\C}_+$ and $\Phi$ is analytic in $\C_+$. Let us first assume that $\Phi:\C_+\to\C_+$. We begin by showing that $\Phi\in\mathcal{G}_{\infty}$ and then deducing the uniform convergence on half-planes $\C_{\varepsilon}$. Hence, choose $m\in\N$, $m\geq2$. We now consider the sequence of functions $\{g_n\}_n$ given by $g_n(s)=m^{-\Phi_n(s)}$. Since, for each $n$, $\Phi_n\in\mathcal{G}_{\infty}$, we deduce that $g_n\in\mathcal{H}^{\infty}$, for every $n\in\N$. Moreover, $\|g_n\|_{\infty}\leq1$ for all $n\in\N$. This is, $\{g_n\}_n$ is a bounded sequence of functions in $\mathcal{H}^{\infty}$. Knowing this, we apply Theorem \ref{th: bayart} to the sequence $\{g_n\}_n$. This gives us the existence of both a subsequence $g_{n_k}$ and a function $g\in\mathcal{H}^{\infty}$ such that
\begin{equation*}
    g_{n_k}\to g,\quad k\to\infty,
\end{equation*}
uniformly in half-planes $\C_{\varepsilon}$, $\varepsilon>0$. On the other hand, by hypothesis, $\Phi_n\to\Phi$ uniformly on compact sets of $\C_+$. This forces $g(s)=m^{-\Phi(s)}$. Then, 
\begin{equation}\label{converce}
    m^{-\Phi_{n_k}}\to m^{-\Phi},\quad k\to\infty,
\end{equation}
uniformly in $\C_{\varepsilon}$. In fact, by the uniqueness of the limit, the whole sequence $\{g_k\}_k$ converges uniformly to $g$. Therefore, we have shown that for every $m\in\N$, $m\geq2$, both $m^{-\Phi_n}\to m^{-\Phi}$ uniformly in $\C_{\varepsilon}$ and $m^{-\Phi}\in\mathcal{H}^{\infty}$. By Theorem \ref{lemma: bay-cas}, $\Phi(s)=c_{\Phi}s+\varphi(s)$, where $c_{\Phi}\in\N\cup\{0\}$ and $\varphi\in\mathcal{D}$. This is, $\Phi\in\mathcal{G}_{\infty}$. It remains to deduce the uniform convergence on half-planes $\C_{\varepsilon}$. Thus, notice that from \eqref{converce} we can deduce that
\begin{equation*}
    2^{\Phi(s)-\Phi_n(s)}\to 1
\end{equation*}
uniformly on vertical strips $\mathbb{S}=\{s\in\C:\alpha<\text{Re}(s)<\beta\}$, $0<\alpha<\beta$, where we have used that the function $2^{\Phi}$ is bounded in any vertical strip $\mathbb{S}$. Indeed, 
\[
|2^{\Phi(s)}|=2^{\text{Re}(\Phi(s))}\leq2^{c_{\Phi}\beta+M}
\]
with $M=\|\varphi\|_{\mathcal{H}^{\infty}(\C_{\alpha})}$, since $\varphi\in\mathcal{H}^{\infty}(\C_{\varepsilon})$, $\varepsilon>0$, (see \cite[Theorem 8.4.1]{queffelecs}). Taking the principal logarithm in any such vertical strip, as the principal logarithm is Lipschitz in a neighbourhood of $1$, we conclude that there exists a sequence of integers $\{t_n\}_n$ such that
\[
(\Phi_n(s)-\Phi(s))\log2+2\pi it_n\to0, \quad n\to\infty, 
\]
uniformly on vertical strips of $\C_{\varepsilon}$. On the other hand, by hypothesis, given a compact set $K\subset\mathbb{S}$, $\Phi_n(s)\to\Phi(s)$ as $n\to\infty$ uniformly in $K$. However, this, together with the fact that $h_n(s)=(\Phi_n(s)-\Phi(s))\log2+2\pi it_n$ tends to zero in any vertical strip $\mathbb{S}$ as $n\to\infty$, forces the sequence $\{t_n\}_n$ to have finitely many non-zero terms. Therefore, we have proved that $\Phi_n\to\Phi$ uniformly on vertical strips of $\C_{\varepsilon}$, $\varepsilon>0$. In particular, this implies that $c_{\Phi_n}=c_{\Phi}$ for $n$ big enough. Indeed, since $\Phi_n\in\mathcal{G}_{\infty}$, for every $n$, and so does $\Phi$, we have that
\[
\Phi_n(s)-\Phi(s)=(c_{\Phi_n}-c_{\Phi})s+\varphi_n(s)-\varphi(s).
\]
The fact that $\Phi_n(s)\to\Phi(s)$ uniformly in $\mathbb{S}$ forces the functions $h_n=\Phi_n-\Phi$ to be bounded in any vertical strip $\mathbb{S}$ for $n$ big enough. Therefore, necessarily, $c_{\Phi_n}=c_{\Phi}$ for $n$ large enough, since $\varphi_n$ and $\varphi$ are bounded on $\mathbb{S}$.

Now, by the convergence of the sequence $\{\Phi_n\}_n$ on vertical strips of $\C_{\varepsilon}$, we have that $\varphi_n\to\varphi$ uniformly on those strips. For every $n$, the Dirichlet series $\varphi_n$ belongs to $\mathcal{H}^{\infty}(\C_{\varepsilon})$, $\varepsilon>0$. Thus the sequence $\{\varphi_n\}_n$ is uniformly Cauchy as $n$ tends to $\infty$ in vertical strips of $\C_{\varepsilon}$ and, then, $\varphi_n$ converges to $\varphi$ uniformly in $\C_{\varepsilon}$, $\varepsilon>0$ thanks to \cite[Section III.5, Exercise 7]{gamelin} (or \cite[Chapter 12, Exercise
9]{rudin}). Consequently, $\Phi_n\to\Phi$ uniformly on $\C_{\varepsilon}$, $\varepsilon>0$. 

\ 
It remains to study the case when $\Phi$ touches the boundary. If this happens, for certain $t\in\R$ we have $\Phi(s)=it$, for all $s\in\C_+$. Regarding the uniform convergence on half-planes, the argument we just carried out also holds for this case.
\end{proof}
\begin{remark}
    This general fact about the closure of the class $\mathcal{G}_{\infty}\cup\{i\R\}$ under the local uniform convergence can be used to give an alternative proof of \cite[Theorem 5.2]{noi}.
\end{remark}

From the very definition of the Gordon-Hedenmalm class $\mathcal{G}$ we can deduce the following immediate consequence of the theorem we just proved.
\begin{corollary}
The class $\mathcal{G}\cup\{s\in\C:\emph{Re}(s)=\frac12\}$ is closed when equipped with the uniform convergence on compact sets.
\end{corollary}
Nonetheless, the class $\mathcal{G}_{\mathcal{A}}$ is no longer `almost' closed under the local uniform convergence in $\C_+$. The reason being that $\overline{\mathcal{G}_{\mathcal{A}}}=\overline{\mathcal{G}_{\infty}}$ (see Theorem \ref{th: ga no cerrada} right below) and that there are non-constant symbols in $\mathcal{G}_{\infty}$ which do not belong to $\mathcal{G}_{\mathcal{A}}$ (Example \ref{example1}).
\begin{theorem}\label{th: ga no cerrada}
$\overline{\mathcal{G}_{\mathcal{A}}}=\overline{\mathcal{G}_{\infty}}$.
\end{theorem}
\begin{proof}
First we see that $\mathcal{G}_{\infty}\subset\overline{\mathcal{G}_{\mathcal{A}}}$. Take $\Phi\in\mathcal{G}_{\infty}$. By Theorem \ref{closeness} we have that the uniform convergence on compacta is equivalent to the uniform convergence in half-planes $\C_{\varepsilon}$, $\varepsilon>0$. Consider the sequence $f_n(s)=\Phi(s+1/n)$, $n\in\N$. Notice that $f_n\in\mathcal{G}_{\mathcal{A}}$ for every $n$. Indeed, $\Phi$ is uniformly continuous in every half-plane $\C_{\varepsilon}$, $\varepsilon>0$ ($\Phi'$ is bounded in every such half-plane), so any horizontal translation $\Phi_{\tau}$, $\tau>0$, is uniformly continuous in $\C_+$. Since $f_n\to\Phi$ uniformly in $\C_{\varepsilon}$, we have that $\Phi\in \overline{\mathcal{G}_{\mathcal{A}}}$. The fact that $\mathcal{G}_{\mathcal{A}}\subset\overline{\mathcal{G}_{\infty}} $ follows immediately from the definition of the classes $\mathcal{G}_{\mathcal{A}}$ and $\mathcal{G}_{\infty}$.
\end{proof}

\section{Semigroups of analytic functions}\label{sec:semanfunc}
The study of semigroups of analytic functions in  the unit disc $\D$ (and then in the right half-plane) started in the early 1900s. The work \cite{porta}, appeared in 1978, due to Berkson and Porta, meant a renewed interest in the study of semigroups of analytic functions. The state of the art can be seen in \cite{manoloal}. We recall the definition straightaway.
\begin{definition} \label{def:semigroup}
We say that a family $\{\Phi_t\}_{t\geq0}$ (in short, $\{\Phi_t\}$) of analytic functions $\Phi_t:\C_{+}\to\C_{+}$ is a semigroup if the following two algebraic properties are satisfied:
\begin{itemize}
\item $\Phi_0(s)=s.$
    \item For every $t,u\geq0$, $\Phi_t\circ\Phi_u(s)=\Phi_{t+u}(s)$.
    \end{itemize}
If, in addition, $\Phi_t$ converges to $\Phi_0$ uniformly on compact subsets of $\C_{+}$ as $t\to0^+$, we say that $\{\Phi_t\}$ is a continuous semigroup.
\end{definition}
One of the goals in this paper is the study of continuous semigroups of analytic functions $\{\Phi_{t}\}$ in the class $\mathcal{G}_{\mathcal{A}}$ and its connection to the strongly continuous semigroups of composition operators in $\mathcal{A}(\C_+)$. See \cite{noi} for the study of this relation between semigroups of composition operators in the Hardy spaces of Dirichlet series $\mathcal{H}^p$, $p\geq1$, and semigroups of analytic functions in the classes $\mathcal{G}$ and $\mathcal{G}_{\infty}$.

In \cite[Proposition 3.2]{noi}, it was proven that for any continuous semigroup in $\mathcal{G}_{\infty}$, we have that $\Phi_t(s)=s+\varphi_t(s)$, $s\in\C_+$. That is, $c_{\Phi_t}=1$ for every $t$. Since $\mathcal{G}_{\mathcal{A}}\subset\mathcal{G}_{\infty}$, the same conclusion holds for continuous semigroups $\{\Phi_t\}$ in $\mathcal{G}_{\mathcal{A}}$. Notice that this result implies that continuous semigroups of analytic functions in the class $\mathcal{G}_{\infty}$ coincide with the continuous semigroups in $\mathcal{G}$.

\ 
 If $\{\Phi_{t} \}$ is a semigroup in the class $\mathcal G_{\infty}$, then $\text{Re}( \Phi_{t}(s))\geq \text{Re}(s)$ for all $s$. This clearly implies that, unless $\{\Phi_{t}\}$ is the trivial semigroup, $\{\Phi_t\}$ has no fixed points in $\C_{+}$ and, in fact, by \cite[Theorem 8.3.1]{manoloal}, the Denjoy-Wolff point of the semigroup is $\infty$, that is,
\begin{equation}\label{DW}
\lim_{t\rightarrow +\infty }\Phi_{t}(s)=\infty,\quad s\in\C_{+}.
\end{equation}

\ 

Berkson and Porta also proved in \cite{porta} the existence of the following limit
\begin{equation}\label{stress}
H(s)=\lim_{t\to0^+}\frac{\Phi_t(s)-s}{t}, \qquad \textrm{ for all } s\in \C_{+}
 \end{equation}
and such limit is uniform on compact sets of $\C_{+}$. In particular, $H$ is holomorphic.
Moreover, $t\mapsto \Phi _{t}(s)$ is the solution of the Cauchy problem:
\begin{equation}\label{cauchy}
\frac{\partial \Phi _{t}(s)}{\partial t}=H(\Phi _{t}(s))\quad
\mbox{and} \quad \Phi _{0}(s)= s\in \C_{+}.
\end{equation}
The function $H$ is called the  {\sl infinitesimal generator} of
the semigroup $\left\{ \Phi _{t}\right\} .$ 
In fact, they also proved, see \cite[Theorem 2.6]{porta}, that $H$ is the infinitesimal generator of a continuous semigroup of analytic functions with Denjoy-Wolff point $\infty$ if and only if $H(\C_{+})\subset\overline{ \C_{+}}$. 

\ 
Since the Denjoy-Wolff point of a semigroup in the Gordon-Hedenmalm class is $\infty$ its infinitesimal generator is a holomorphic function sending the right half-plane into its closure. In \cite[Theorem 5.1]{noi}, a description of the infinitesimal generators of the continuous semigroups in $\mathcal{G}_{\infty}$ was given.
\begin{theorem}\label{gcoro}
Let $H:\C_+\to\overline{\C}_+$ be analytic. Then, the following statements are equivalent:
\begin{enumerate}[a)]
    \item $H$ is the infinitesimal generator of a continuous semigroup of elements in the class $\mathcal{G}$.
    \item $H$ is bounded in every half-plane $\C_{\varepsilon}$ and $H\in\mathcal{D}$.
    \item $H\in\mathcal{D}$.
\end{enumerate}
\end{theorem}
Using this description of the infinitesimal generators of continuous semigroups in the class $\mathcal{G}$, the three authors proved in \cite{noi1} the following result which will play a key role in the proof of the main theorem from Section \ref{sec:semigroupsinalg}. 
\begin{theorem}\label{convunifsemg}
Let $\{\Phi_{t}\}$ be a continuous semigroup in the Gordon-Hedenmalm class $\mathcal{G}$, then $\{\Phi_{t}\}$ converges to the identity map uniformly in the right half-plane, as $t$ goes to zero.
\end{theorem}

\section{Semigroups of composition operators in $\mathcal{A}(\C_+)$}\label{sec:semigroupsinalg}
There are numerous areas in Analysis where the theory of strongly continuous semigroups of bounded operators on Banach spaces has proven to be quite fruitful. We recall this  notion. 
\begin{definition}
Let $X$ be a Banach space and $\{T_{t}\}_{t\geq0}$ (in short, $\{T_t\}$) a family of bounded operators from $X$ into itself. We say that  $\{T_{t}\}$ is a semigroup if the following two algebraic properties are satisfied:
\begin{enumerate}
\item[(i)] $T_0=\mathrm{Id}$, where $\mathrm{Id}$ denotes the identity map on $X$;
    \item[(ii)] For every $t,u\geq0$, $T_t\circ T_u=T _{t+u}$.
    \end{enumerate}
If, in addition, it satisfies that 
\begin{enumerate}
\item[(iii)] $\lim_{t\to 0^{+}}T_{t}f=f$ for all $f\in X$
    \end{enumerate} we say that it is a strongly continuous semigroup (also known as $C_{0}$-semigroup).
\end{definition}
It is well-known that (iii) is equivalent to the fact that, for each $f\in X$, the map $[0,+\infty)\ni t\mapsto T_{t}f$ is continuous \cite[Page 3, Proposition 1.3]{engels}.
Clearly, if we have a semigroup of analytic functions $\{\Phi_{t}\}$, we obtain a semigroup of composition operators $\{C_{\Phi_{t}}\}$, whenever such composition operators are well-defined. Conversely, if $T_t=C_{\Phi_t}$ is a semigroup of bounded operators in $\mathcal{A}(\C_+)$, $\{\Phi_t\}$ is a semigroup in $\mathcal{G}_{\mathcal{A}}$. This was established  in the setting of $\mathcal{H}^p$ spaces, see \cite[Proposition 4.3]{noi} for a proof of this statement. The same proof holds for the algebra $\mathcal{A}(\C_+)$. 

\ 
The relationship between semigroups of analytic functions in $\mathcal{G}_{\mathcal{A}}$ and semigroups of composition operators in $\mathcal{A}(\C_+)$ is not only algebraic, but also topological. As in the case of the $\mathcal{H}^p$ spaces, for $1\leq p<\infty$ (\cite[Theorem 4.6]{noi}), there exists a one-to-one correspondence between each other. This first result clearly holds in virtue of \cite[Theorem 4.4]{noi}, since any strongly continuous semigroup in $\mathcal{A}(\C_+)$ is strongly continuous in $\mathcal{H}^2$.
\begin{proposition}\label{Ttfuertcont}
    Let $\{\Phi_t\}$ be a semigroup in $\mathcal{G}_{\mathcal{A}}$ and $T_t=C_{\Phi_t}$ for each $t$. If $\{T_t\}$ is strongly continuous in $\mathcal{A}(\C_+)$, then $\{\Phi_t\}$ is continuous.
\end{proposition}
The following result is a general fact about symbols in $\mathcal{G}_{\mathcal{A}}$ which will be needed in the proof of Theorem \ref{teor:fcont-uniform}.
\begin{proposition}\label{prop:acotado implica cu}
    Let $\Phi$ belong to $\mathcal{G}_{\mathcal{A}}$ with $\Phi(s)=c_{\Phi}s+\varphi(s)$. If $\varphi\in\mathcal{H}^{\infty}$, then $\Phi$ is uniformly continuous in $\C_+$.
\end{proposition}
\begin{proof}
Since $\Phi$ is in $\mathcal{G}_{\mathcal{A}}$, then it is uniformly continuous in the sets
\[
A_M=\{s\in\C:0<\text{Re}(\Phi(s))<M\},\quad \forall M>0.
\]
We claim that there exists $M_0>0$ such that
$\Omega_{1/2}=\{s\in\C:0<\text{Re}(s)<1/2\}\subset A_{M_0}$. Indeed, given $s\in\Omega_{1/2}$, we have that $\text{Re}(\Phi(s))\leq c_{\Phi}/2+\|\varphi\|_{\infty}=K$. Taking $M_0=K$ the claim follows. Then, $\Phi$ is uniformly continuous in $A_{M_0}$ and, in particular, in $\Omega_{1/2}$. The function $\Phi$ is also uniformly continuous in $\C_{1/4}$ (Corollary \ref{coro:cont unif}). Hence, we deduce that $\Phi$ is uniformly continuous in the right half-plane $\C_+$. 
\end{proof}

\begin{theorem}\label{teor:fcont-uniform}
Let $\{\Phi_t\}$ be a continuous semigroup of analytic functions in $ \mathcal{G}_{\infty}$ and $T_t=C_{\Phi_t}$, $t\geq0$. Then, the following statements are equivalent:
\begin{enumerate}[a)]
    \item The semigroup $\{T_t\}$ is strongly continuous in $\mathcal{A}(\C_+)$.
   \item $\Phi_t\in\mathcal{G}_{\mathcal{A}}$, $t>0$.
    \item $\varphi_t\in\mathcal{A}(\C_+)$, $t>0$.
\end{enumerate}
\end{theorem}
\begin{proof}
 By Theorem \ref{thsimal}, we have that $a)$ implies $b)$. 

\ 
By the definition of the class $\mathcal{G}_{\mathcal{A}}$, clearly, $c)$ implies $b)$. We now prove the reverse implication. We know that $\Phi_t(s)=s+\varphi_t(s)$. By Theorem \ref{convunifsemg}, there exists a $t_0>0$ such that $\varphi_t\in\mathcal{H}^{\infty}$ for all $t\in[0,t_0]$. By the semigroup structure, we can extend the conclusion to the range $t>0$. Therefore, every function $\Phi_t$ is under the hypothesis of Proposition \ref{prop:acotado implica cu}. Applying this result, we conclude that $\Phi_t$ is uniformly continuous in $\C_+$ for every $t>0$


\
For the proof of $b)$ implies $a)$ notice that the hypothesis together with Theorem \ref{thsimal} imply that $T_tf=f\circ\Phi_t$ belongs to $\mathcal{A}(\C_+)$ for any $f$ in this space. Theorem \ref{convunifsemg} yields that $\Phi_t(s)\to s$ uniformly in $\C_+$. Furthermore, every Dirichlet series $f$ in $\mathcal{A}(\C_+)$ is uniformly continuous in $\C_+$. Consequently, we conclude that
$f\circ\Phi_t\to f$ uniformly in $\C_+$ for every $f\in\mathcal{A}(\C_+)$.
\end{proof}
There exists an analogue version of Theorem \ref{teor:fcont-uniform} in the classical setting of the unit disc algebra $A(\D)$ (see \cite[Theorem 1.2]{Contreras-Diaz1}).

\begin{remark}
It is said that $\{T_t\}$ is uniformly continuous if and only if $T_{t}$ converges, as $t$ goes to $0$, to the identity map in the norm of the space of linear bounded operators in $X$.  Clearly, every uniformly continuous semigroup is strongly continuous. Using essentially the same argument as in \cite[Theorem 5.8]{noi} as well as the fact that the multiplication operator $f\mapsto n^{-s}f$ is an isometry in $\mathcal{H}^{\infty}$, we can conclude that the only uniformly continuous semigroup of composition operators in $\mathcal{A}(\C_+)$ is the trivial one, namely, $T_t=\text{Id}$ for every $t$.
\end{remark}
\subsection{The Koenigs function of semigroups in $\mathcal{G}_{\mathcal{A}}$}
The next pages are devoted to the study of some properties of the Koenigs function of continuous semigroups in the class $\mathcal{G}_{\mathcal{A}}$. The same thing was done previously for the class $\mathcal{G}$ in \cite{noi}. In fact, using some of the properties established there, we are able to construct continuous semigroups of analytic functions in $\mathcal{G}_{\infty}$ but not in $\mathcal{G}_{\mathcal{A}}$. The reader interested in reading about this topic in the more classical setting of continuous semigroups in the unit disc $\D$ is addressed to the book \cite[Chapter 9]{manoloal}.

\ 
Before moving ahead, we provide some definitions that will be needed later.
\begin{definition}
    Let $\lambda\in\C$ be such that $\text{Re}(\lambda)\geq0$ and $\Omega$ a domain in $\C$ such that $0\in\Omega$. 
    \begin{itemize}
        \item[a)] We say that $\Omega$ is $\lambda$-spirallike if $e^{-\lambda t}\Omega\subset\Omega$ for all $t\geq0$.
        \item[b)] If $\lambda\in(0,\infty)$, we say that $\Omega$ is starlike
        \item[c)] A map $h\colon\D\to\C$ is $\lambda$-spirallike with respect to $\tau\in\D$ if it is univalent, $h(\tau)=0$ and $h(\D)$ is a $\lambda$-spirallike domain.
        \item[d)] If a map $h\colon\D\to\C$ is $\lambda$-spirallike with $\lambda\in(0,\infty)$, it is also called starlike.
    \end{itemize}    
\end{definition}
Next Theorem is proven in the unit disc setting in \cite[Theorem 9.3.5]{manoloal}, but it holds in any simply connected domain $\Omega$.
\begin{theorem}\label{thkoenigsa}
Let $\{\Phi_t\}$ be a non-elliptic continuous semigroup of analytic functions in $\C_+$. Then, there exists a  univalent function $h:\C_+\to\C$ such that
\begin{equation}\label{modeloa}
    h\circ\Phi_t(s)=h(s)+t.
\end{equation}
The function $h$ is unique up to an additive constant. 
\end{theorem}
The function $h$ is known as the \emph{Koenigs function} of the semigroup. The interest about studying properties of the Koenigs function $h$ lies on the useful information it can provide about the semigroup. In fact, since the  function $h$ is univalent, we can recover the semigroup $\{\Phi_t\}$ as
\begin{equation*}
    \Phi_t(s)=h^{-1}(h(s)+t).
\end{equation*}
Now, a derivation with respect to $t$ and evaluating at $t=0$ in \eqref{modeloa} gives that
\begin{equation*}
    h'(s)=\frac{1}{H(s)},
\end{equation*}
where $H$ is the infinitesimal generator of the continuous semigroup $\{\Phi_t\}$. Therefore, the infinitesimal generator of a continuous semigroup $\{\Phi_t\}$ and its Koenigs function are linked via the latter identity.

\ 
In the setting of continuous semigroups in the class $\mathcal{G}$, there exists a full description of the Koenigs functions of such semigroups given in \cite[Corollary 6.8]{noi}. We recall it straightaway.
\begin{theorem}\label{th:Koenigs}
Let $h$ be a holomorphic function in $\C_{+}$. Then, $h$ is the Koenigs function of a continuous semigroup in $\mathcal G$ if and only if $h'$ is a Dirichlet series satisfying $h'(\C_{+})\subset \overline {\C_{+}}\setminus \{0\}$. 
\end{theorem}
Take $\{\Phi_t\}$ a continuous semigroup in $\mathcal{G}_{\infty}$. If for certain $b\in\R$ the limit
    \[
    \lim_{s\to ib}\Phi_t(s)
    \]
happens to exist, then it is finite. Indeed, thanks to Theorem \ref{convunifsemg}, the elements of the semigroup are of the form $\Phi_t(s)=s+\varphi_t(s)$, where $\varphi_t\in\mathcal{H}^{\infty}$. Therefore, if the limit of $\Phi_t$ at $\partial\C_+$ exists, it is necessarily finite.
    
\begin{theorem}\label{koenigs}
Let $\{\Phi_t\}$ be a continuous semigroup in $\mathcal{G}_{\infty}$ and $h$ be its Koenigs function. Then, the following assertions are equivalent:
\begin{enumerate}[a)]
   \item For every $t>0$ and for every $b\in\R$, the limit $\lim_{s\to ib}\Phi_t(s)$ exists.
   \item There exists $t_0>0$ such that for every $b\in\R$ the limit $\lim_{s\to ib}\Phi_{t_0}(s)$ exists.
  \item  For every $\beta\in \partial\C_+\cup\{\infty\}$, the limit $\lim_{s\to \beta}h(s)$ exists in the Riemann sphere. Equivalently, $h$ has a continuous extension from $\overline{\C_+}\cup\{\infty\}$ to the Riemann sphere.
   \item $\Phi_t$ is uniformly continuous on $C_{R}=\{z\in\C_+:|\emph{Im}(z)|<R\}$, for every $R>0$ and $t>0$. 
 \end{enumerate}
\end{theorem}
\begin{proof}
$a)$ implies $b)$ is trivial. For $b)$ implies $c)$, consider $L:\D\to\C_+$ the Riemann mapping $L(z)=\frac{1+z}{1-z}$. Then, define the analytic functions in the unit disc $\D$ given by $\Psi_{t}(s)=L^{-1}\circ\Phi_{t}\circ L(z)$, $z\in\D$. Clearly, $\{\Psi_t\}$ is a continuous semigroup in $\D$. In fact, its Koenigs function is $h_{\D}=h\circ L$, where $h$ stands for the Koenigs function of the semigroup $\{\Phi_t\}$. Moreover, its Denjoy-Wolff point is $1$. Since $L:\overline{\D}\to\overline{\C_+}$ is a homeomorphism, to prove $c)$ it suffices to check that the limit $\lim_{z\to\tau}h_{\D}(z)$ exists in the Riemann sphere for every $\tau\in\partial\D$. Thanks to $b)$, the function $\Psi_{t_0}:\D\to\D$ has limit at every point $\overline{\D}\setminus\{1\}$. However, $\Phi_{t_0}$ also has a limit at the point $z=1$, since $\Phi_{t_0}\in\mathcal{G}_{\infty}$, so $\Phi_{t_0}(s)=s+\varphi_{t_0}(s)$, with $\varphi_{t_0}\in\mathcal{H}^{\infty}$. Therefore,
$\lim_{s\to\infty}\Phi_{t_0}(s)=\infty$, and then,
\[
\lim_{z\to1}\Psi_{t_0}(z)=1.
\]
Then, we have proven that $\Psi_{t_0}$ belongs to the unit disc algebra $A(\D)$. Therefore, by \cite[Theorem 11.3.8, $c)\Rightarrow a)$]{manoloal}, we conclude that the unrestricted limit of $h_{\D}$ exists at every point of $\partial\D\setminus\{1\}$. By the definition of $h_{\D}$, we have that $h$ has a continuous extension to the imaginary axis $i\R$. It remains to prove the existence of its extension to the Denjoy-Wolff point of the semigroup $\{\Phi_t\}$. If the semigroup $\{\Phi_t\}$ consists of automorphisms, the result trivially holds. Therefore, assume that they are not automorphisms. Now, by \cite[Proposition 6.10]{noi}, we know that the semigroup $\{\Phi_t\}$ has zero-hyperbolic step. Thanks to the isometric nature of the hyperbolic distance for biholomorphisms, see \cite[Proposition 1.3.10]{manoloal}, we conclude that the semigroup $\{\Psi_t\}$ also has zero hyperbolic step. An application of Theorem \cite[Theorem 11.1.4, 3) iii)]{manoloal} yields
\[
\lim_{z\to1}h_{\D}(z)=\infty.
\]
This guarantees that the Koenigs function $h_{\D}$ of the semigroup $\{\Psi_t\}$ has unrestricted limit at the Denjoy-Wolff point $\tau=1$. Since $h=h_{\D}\circ L^{-1}$, we conclude that
\[
\lim_{s\to\infty}h(s)=\lim_{s\to\infty}h_{\D}\circ L^{-1}(s)=\infty.
\]
For $c)$ implies $a)$, we define the semigroup $\Psi_t$ analogously and we let $h_{\D}$ stand for its Koenigs function. Thanks to $c)$, we have that $h_{\D}$ has unrestricted limit at $\D\setminus\{1\}$, and the Denjoy-Wolff of the semigroup is $1$. Using  \cite[Theorem 11.3.8, $a)\Rightarrow b)$]{manoloal}, we conclude that $\Psi_t$ belongs to the disc algebra $A(\D)$ for every $t>0$. Since, $\Phi_t(s)=L\circ\Psi_t\circ L^{-1}(s)$, we have that for every $b\in\R$, the limit $\lim_{s\to ib}\Phi_t(s)$ exists. 

\ 
If we assume $d)$ then, for any $b\in\R$, there exists $R>0$ big enough so that $\varepsilon+ib$,  belongs to $C_R$ for every $\varepsilon>0$. Since $\Phi_t$ is uniformly continuous in $C_R$, we can extend it to $\overline{C}_R$ and, in particular, to the point $ib$ and $a)$ follows. Reciprocally, suppose that $a)$ holds. Take $R>0$. Then, $\Phi_t$ can be extended to $\overline{C}_R$, by hypothesis. If we restrict ourselves to those points in $\overline{C}_R$ with real part less or equal than $1/2$, we have that, by compactness, $\Phi_t$ is uniformly continuous there. The uniform continuity of $\Phi_t$ in every half-plane $\C_{\varepsilon}$, $\varepsilon>0$ yields the result.
\end{proof}
\begin{remark}
    Notice that if $\{\Phi_t\}$ is a continuous semigroup in the class $\mathcal{G}_{\infty}$ such that there exists $t_0>0$ satisfying $\Phi_{t_0}\in\mathcal{G}_{\mathcal{A}}$, then all the conditions from Theorem \ref{koenigs} are satisfied. 
\end{remark}
In the next Corollary we show that, in analogy with case of continuous semigroups of holomorphic self maps of the unit disc (\cite[Theorem 11.3.8]{manoloal}), for the case of a certain class of continuous semigroups in $\mathcal{G}_{\infty}$, it suffices that an element of the semigroup belongs to $\mathcal{G}_{\mathcal{A}}$ to conclude that every function of the semigroup is in the class $\mathcal{G}_{\mathcal{A}}$.
\begin{corollary}\label{coro1}
  Let $\{\Phi_t\}$ be a continuous semigroup in $\mathcal{G}_{\infty}$. Assume the existence of $t_0>0$ such that $\Phi_{t_0}\in\mathcal{G}_{\mathcal{A}}$ and $t_1>0$ such that $\Phi_{t_1}(s)=s+g(k^{-s})$, with $k\in\N$, $k\geq2$. Then, $\varphi_{t_1}\in\mathcal{A}(\C_+)$. In particular, $\Phi_{t_1}\in\mathcal{G}_{\mathcal{A}}$
\end{corollary}
\begin{proof}
    If there exists $t_0>0$ such that $\Phi_{t_0}$ is in $\mathcal{G}_{\mathcal{A}}$, the conditions from Theorem \ref{koenigs} are all satisfied. Hence, we have that for every $t$, $\Phi_t$ is uniformly continuous on the closure of $C_R\cap\{0<\text{Re}(z)< 1/2\}$, for every $R>0$ and the same occurs for $\varphi_t$. Since $\varphi_t$ is periodic (with an imaginary period), we conclude the uniform continuity of $\varphi_t$ in the whole right half-plane and the claim is proven. 
\end{proof}
\begin{corollary}\label{coro2}
    Let $\{\Phi_t\}$ be a continuous semigroup in $\mathcal{G}_{\infty}$ and $h$ its Koenigs function.  Assume that $\varphi_t(s)=g_t(k^{-s})$, $g_t:\D\to\C_+$ holomorphic, $k\in\N$, $k\geq2$ and that $h$ is uniformly continuous in $\C_+$. Then, $\Phi_t\in\mathcal{G}_{\mathcal{A}}$ for all $t$.
\end{corollary}
\begin{proof}
    Since $h$ is uniformly continuous in $\C_+$, by Theorem \ref{koenigs}, every $\Phi_t(s)=s+\varphi_t(s) $ is uniformly continuous in any horizontal strip of $\C_+$ of width $R>0$. Then, by the periodicity of $\varphi_t$, every symbol $\Phi_t$ is uniformly continuous on $\C_+$ and the claim follows.
\end{proof}
In \cite[Theorem 5.2]{noi}, it was shown that given a continuous semigroup $\{\Phi_t\}$ in $\mathcal{G}_{\infty}$ such that $\Phi_t(s)=s+g_t(k^{-s})$, $g_t:\D\to\C_+$ holomorphic, then the infinitesimal generator $H$ of the semigroup is of the form $H(s)=G(k^{-s})$, where $G:\D\to\C_+$. Clearly, this also gives that the Koenigs function $h$ of the semigroup is  of the form $h(s)=F(k^{-s})$, where $F:\D\to\C_+$ is holomorphic, too. The reverse implication also holds. More precisely, for $\{\Phi_t\}$ a continuous semigroup in the class $\mathcal{G}_{\infty}$ such that its Koenigs function $h$ is given by $h(s)=F(k^{-s})$, with $F:\D\to\C_+$ holomorphic and $k\in\N$, $k\geq2$, the elements $\Phi_t$ of the semigroup are all of the form $\Phi_t(s)=s+\varphi_t(s)$, where $\varphi_t(s)=g_t(k^{-s})$, and $g_t:\D\to\C_+$ is holomorphic for every $t$.

\ 
In fact, these specific semigroups can be fully described through spirallike functions on $\D$.
\begin{proposition}\label{prop:caracsemperio}
Let $k\in\N$, $k\geq2$.
\begin{itemize}
    \item[a)] Let $f:\D\to\C$ be $c$-spirallike with respect to the origin. There exists a continuous semigroup $\{\Phi_t\}$ in $\mathcal{G}_{\infty}$ such that its Koenigs function is 
    \[
    h(s)=-\frac{1}{c\log k}\log(f(k^{-s})), \quad c>0.
    \]
    Reciprocally, let $\{\Phi_t\}$ be a continuous semigroup in $\mathcal{G}_{\infty}$ such that its Koenigs function $h$ is of the form $h(s)=ds+g(k^{-s})$, where $g\colon\D\setminus\{0\}\to\C$ is holomorphic, $d\in\overline{\C}_+\setminus\{0\}$. Then, $f(z)=z\exp(-\log(k) g(z)/d)$ is $1/d$-spirallike with $d\in\overline{\C}_+\setminus\{0\}$. 
    \item[b)] Moreover, if $f$ cannot be continuously extended to $\partial\D$, then, for every $t$, $\Phi_t\not\in\mathcal{G}_{\mathcal{A}}$.
 \end{itemize}
  
\end{proposition}
\begin{proof}
We begin showing the necessity in $a)$. Since $f$ is $c$-spirallike with respect to zero, $f(0)=0$. Therefore, the function $F(s)=f(k^{-s})$ does not vanish in $\C_+$ and, consequently, it admits a logarithm. Set $h(s)=-\frac{1}{c\log k}\log(f(k^{-s})), \text{Re}(c)\geq0, c\not=0$. We claim that it is the Koenigs function of a continuous semigroup in $\mathcal{G}_{\infty}$. By Theorem \ref{th:Koenigs}, it suffices to check that $h'$ is a Dirichlet series such that $\text{Re}(h')>0$. A simple computation shows that
     \[
    h'(s)=\frac{k^{-s}}{c}\frac{f'(k^{-s})}{f(k^{-s})}.
    \]
Since $f:\D\setminus\{0\}\to\D\setminus\{0\}$ and $k^{-s}$ sends the right half-plane into the punctured unit disc, we have that $h'=G(k^{-s})$ with $G$ a holomorphic function in $\D$. Therefore, $h'$ is a Dirichlet series. Now, as $f$ is $c$-spirallike, we conclude that $\text{Re}(h')>0$ (\cite[Theorem 9.4.5]{manoloal}) and Theorem \ref{th:Koenigs} gives the desired conclusion.

\ 
For the sufficiency of $a)$, since $\{\Phi_t\}$ is a continuous semigroup in $\mathcal{G}_{\infty}$, by \cite[Corollary 6.8]{noi}, we have that $h'(s)=d-\log(k)g'(k^{-s})k^{-s}$ is a Dirichlet series such that $\text{Re}(h')>0$. Taking this into account, a simple computation for those $z\in\D$ such that $f(z)\not=0$ yields
\[
\text{Re}\left(
dz\frac{f'(z)}{f(z)}
\right)
=
\text{Re}\left(
d-\log(k)g'(z)z
\right)>0,
\]
where we are taking $z=k^{-s}$, $s\in\C_+$.

\ 
Regarding part $b)$, if $f$ fails to have a continuous extension to $\partial\D$, then so does the Koenigs function $h$. An application of Theorem \ref{koenigs} yields the desired result.
    \end{proof}
This description of these `periodic' continuous semigroups in $\mathcal{G}_{\infty}$ will allow us to provide many examples of continuous semigroups in $\mathcal{G}_{\infty}$ which fail to be in $\mathcal{G}_{A}$.
\begin{theorem}\label{th: sem ginf no ga}
There exists a continuous semigroup $\{\Phi_t\}$ in the class $\mathcal{G}_{\infty}$ such that $\Phi_t$ is not in $\mathcal{G}_{\mathcal{A}}$ for any $t>0$.
\end{theorem}
\begin{proof}
For $n\in\N$, consider the segments
    $\gamma_n=\{z\in\D: 1/2\leq|z|<1, \  \text{arg}(z)=\frac{\pi}{2n}\}$ and let $\gamma_0=[1/2,1)$. We set 
    $\Gamma=\bigcup_{n\geq0}\gamma_n$.  
    Then, the domain $\Omega=\D\setminus \Gamma$ is starlike with respect to the origin. We shall now consider the Riemann map $f$ from the unit disc $\D$ into the domain $\Omega$ such that $f(0)=0$. By Proposition \ref{prop:caracsemperio} $a)$, there exists a continuous semigroup $\{\Phi_t\}$ in $\mathcal{G}_{\infty}$ whose Koenigs function is
    \[
    h(s)=-\frac{1}{c\log k}\log(f(k^{-s})), \quad c>0.
    \]
 Now, by Carathéodory's Theorem (\cite[Theorem 4.3.1]{manoloal}), since $\partial\Omega$ is not locally connected, $f$ fails to have a continuous extension to $\partial\D$.  Then, Proposition \ref{prop:caracsemperio} $b)$ gives the conclusion.   
  \end{proof}
We can also provide examples of a well known fact for semigroups of analytic functions in the unit disc: there exist continuous semigroups $\{\Phi_t\}$ in $\D$ such that for some $t_0>0$ the composition operators $C_{\Phi_{t}}$ are compact in $H^{\infty}(\D)$ for every $t> t_0$ but fail
to be compact in $ H^{\infty}(\D)$ for $t\in[0,t_0]$. 
\begin{proposition}
    There exists a continuous semigroup $\{\Phi_t\}$ in $\mathcal{G}_{\mathcal{A}}$ such that for some $t_0>0$ the composition operators $C_{\Phi_t}$ are compact in $\mathcal{A}(\C_+)$ for $t>t_0$ but $C_{\Phi_t}$ fails to be compact in $\mathcal{A}(\C_+)$ for every $t\in[0,t_0]$.
\end{proposition}
\begin{proof}
    Let $D_1=\D\setminus [1/2,1)$ be the starlike domain with respect to the origin. Let $f:\D\to D_1$ be the Riemann mapping fixing the origin and such that $f'(0)>0$. In particular, this implies that $f((-1,1))\subset (-1,1/2)$. Now, thanks to Proposition \ref{prop:caracsemperio} $a)$, there exists a continuous semigroup $\{\Phi_t\}$ in $\mathcal{G}_{\infty}$ whose Koenigs function is 
    \[
    h(s)=-\frac{1}{c\log 2}\log(f(2^{-s})), \quad c>0.
    \]
   Observe that $f$ can be continuously extended to $\overline{\D}$ (by Carathéodory's Theorem or by direct verification). This implies, by Theorem \ref{koenigs}, that for every $t>0$, $\Phi_t(s)=s+\varphi_t(s)$ is uniformly continuous on horizontal strips of $\C_+$ (and so is $\varphi_t)$. By the $\frac{2\pi i}{\log2}$-periodicity of $\varphi_t$, we conclude that for all $t>0$, $\Phi_t$ is uniformly continuous in $\C_+$. Hence, $\{\Phi_t\}\in\mathcal{G}_{\mathcal{A}}$. Now, $h$ maps $\C_+$ into $D_2$, where 
$D_2=\C_+\setminus\bigcup_{k\in\Z}\{t+2k\pi i: t\in[0,1/c]\}$. Suppose that $t>1/c$. Then,
\[
\Phi_t(\C_+)=h^{-1}(h(\C_+)+t)\subset h^{-1}(\C_t).
\]
We claim that there exists $\eta>0$ such that $h^{-1}(\C_t)\subset \C_{\eta}$. If this were not the case, there would exists a sequence $\{s_n\}\in\C_t$ such that $\text{Re}(s_n)\to0$ and $\text{Re}(h(s_n))>t$, for all $n$. For each $n$, set $z_n=2^{-s_n}$. Then, $|z_n|\to1$. Now,
\[
|f(z_n)|=\exp(-c\log(2)\text{Re}(h(s_n)))<\exp(-ct\log2)=2^{-tc}<1/2.
\]
Nonetheless, since $|z_n|\to1$, necessarily $f(z_n)\to\partial D_1$, which constitutes a contradiction with the fact that $|f(z_n)|<2^{-tc}<1/2$.

\ 
Suppose now that $t\leq 1/c$. Let $\mathbb{S}=\{z\in\C_+: 0<\text{Im}(z)<2\pi i/\log(2)\}$. The map $\phi(s)=f(2^{-s})$ is a bijection between the set $\mathbb{S}
$ and $\D\setminus[0,1)$. Therefore, there exists a sequence $\{s_n\}$ in $\{z:\text{Im}(z)>0\}\cap\mathbb{S}$ such that $f(2^{-s_n})\to1$. Therefore, considering the principal branch of the logarithm of $f(2^{-s_n})$, we have that $h(s_n)\to0$. Hence, $h(s_n)+t\to t\in\partial D_1$ as $n\to\infty$. Then, $h^{-1}(h(s_n)+t)\to\partial\C_+$ as $n\to\infty$. We conclude that $\text{Re}(\Phi_t(s_n))\to0$. Hence, by Theorem \ref{thcomp}, the semigroup $\{\Phi_t\}$ satisfies the desired properties with $t_0=1/c$.
\end{proof}


\subsection{Some final examples}
We begin giving a condition on the infinitesimal generator $H$ so that the corresponding semigroup $\{\Phi_t\}$ lies in $\mathcal{G}_{\mathcal{A}}$.
\begin{lemma}
Let $H$ be a holomorphic function in $\C_+$ such that $H\in\mathcal{D}$ and $H'\in\mathcal{H}^{\infty}$. Then, $H\in\mathcal{A}(\C_+)$.
\end{lemma}
\begin{proof}
Since $H'$ is bounded, $H$ is, in fact, Lipschitz-continuous in $\C_+$ and, consequently, uniformly continuous in $\C_+$. Hence, $H\in\mathcal{A}(\C_+)$.
\end{proof}
\begin{proposition}\label{prop:derivada acotada}
Let $H:\C_+\to\C_+$ be in $\mathcal{H}^{\infty}$ and such that $H'\in\mathcal{H}^{\infty}$. Then, the corresponding semigroup $\{\Phi_t\}_t$ is in $\mathcal{G}_{\mathcal{A}}.$
\end{proposition}
\begin{proof}
The claim follows immediately from Gronwall's Lemma.  Indeed, set $M:=\|H'\|_{\infty}$, consider $s_1,s_2\in\C_+$ and let $G(t)=|\Phi_t(s_1)-\Phi_t(s_2)|$. Then, for $t>0$ fixed,
\begin{align*}
    G(t)=|\Phi_t(s_1)-\Phi_t(s_2)|&\leq |s_1-s_2|+\int_0^t|H(\Phi_{\tau}(s_1))-H(\Phi_{\tau}(s_2))|d\tau\\
    &\leq G(0)+ M\int_0^t|\Phi_{\tau}(s_1)-\Phi_{\tau}(s_2)|d\tau\\
    &=G(0)+M\int_0^t|G(\tau)|d\tau.
    \end{align*}
Then, Gronwall's Lemma yields
 \[
    G(t)\leq G(0)\exp(Mt).
\]
Therefore, for each $t>0$, the function $\Phi_t$ is Lipschitz so, in particular, it is uniformly continuous. By Theorem \ref{teor:fcont-uniform} the conclusion follows.
\end{proof}
Throughout the next examples, $\{\Phi_t\}$ will denote a continuous semigroup in the class $\mathcal{G}_{\infty}$, $H$ will stand for its infinitesimal generator and $h$ for its Koenigs function.
\begin{example}\emph{ $\{\Phi_t\}\in\mathcal{G}_{\mathcal{A}}$ needs not imply that $h\in\mathcal{A}(\C_+)$.}

\ 
 Let $h'(s)=\frac{1}{1-2^{-s}}$. By Theorem \ref{th:Koenigs}, its primitive, $h$ is the Koenigs function of a continuous semigroup $\{\Phi_t\}$ in $\mathcal{G}_{\infty}$. Now, since $H=1/h'$ is a Dirichlet polynomial, by Proposition  \ref{prop:derivada acotada}, we conclude that $\{\Phi_t\}$ is actually in $\mathcal{G}_{\mathcal{A}}$. However, when $s=x$, $x\in(0,\varepsilon)$, $\varepsilon>0$, the function $h'$ behaves like $1/x$ and its primitive is clearly unbounded close to zero.
 \end{example}
\begin{example}
    \emph{$\{\Phi_t\}\in\mathcal{G}_{\mathcal{A}}$ needs not imply that $H\in\mathcal{H}^{\infty}$.}

    \ 
Now, we take $h'(s)=1+2^{-s}$. Its primitive is clearly uniformly continuous in $\C_+$. On the other hand, since $H=1/h'=g(2^{-s})$, where $g:\D\to\C_+$ is holomorphic, then the semigroup is of the form $\Phi_t(s)=s+f_t(2^{-s})$, where $f_t:\D\to\C_+$ for every $t>0$. Then, Corollary \ref{coro2} guarantees us that $\{\Phi_t\}\in\mathcal{G}_{\mathcal{A}}$. However, $H(s)=1/(1+2^{-s})$ is not in $\mathcal{H}^{\infty}$.
\end{example}

\begin{example}
    \emph{$H\in\mathcal{H}^{\infty}$ needs not imply that $\{\Phi_t\}\in\mathcal{G}_{\mathcal{A}}$.}

    \ 
    For $n\in\N$, consider the segments
    $\gamma_n=\{z\in\D: 1/2\leq|z|<1, \  \text{arg}(z)=\frac{\pi}{2n}\}$ and let $\gamma_0=[1/2,1)$. We set 
    $\Gamma=\bigcup_{n\geq0}\gamma_n$.  
    Then, the domain $D_1=\D\setminus \Gamma$ is starlike. Consider $f\colon\D\to D_1$ the Riemann mapping fixing the origin. Proposition \ref{prop:caracsemperio} guarantees the existence of a continuous semigroup $\{\Phi_t\}$ in $\mathcal{G}_{\infty}$ whose Koenigs function is 
    \[
    h(s)=-\frac{1}{c\log 2}\log(f(2^{-s})), \quad c>0.
    \]
We consider the function $h_1(s)=h(s)+ds$, where $\text{Re}(d)>0$. Since $\text{Re}(h')>0$ and $h_1'$ is still a Dirichlet series mapping $\C_+$ into itself, we conclude that $h_1$ is the Koenigs function of a continuous semigroup $\{\Psi_t\}$ in $\mathcal{G}_{\infty}$. Clearly, $|h_1'|>\text{Re}(h')>\text{Re}(d)>0$, so $H_1=1/h_1\in\mathcal{H}^{\infty}$. However, since $\partial\Omega$ is not locally connected, $f$ fails to have a continuous extension to $\partial\D$ and, consequently, so does $h_1$. Theorem \ref{koenigs} prevents the semigroup $\{\Psi_t\}$ from having a continuous extension to $\partial\C_+$ and, therefore, to belong to the class $\mathcal{G}_{\mathcal{A}}$.
\end{example}

\noindent 
\textbf{Acknowledgements.} We acknowledge Professor J. Bonet for calling our attention to some useful references.

\end{document}